\documentclass[a4paper,10pt, oneside]{article}
\usepackage{amsmath,amssymb,amsthm,bm,mathrsfs,graphicx}
\usepackage{authblk}
\usepackage[font=small,labelfont=md,textfont=it]{caption}
\usepackage{floatrow,float}
\usepackage[titletoc, title]{appendix}
\usepackage[colorlinks,linkcolor=blue,citecolor=blue]{hyperref}
\usepackage{etoolbox}
\usepackage{longtable}
\usepackage{diagbox}
\usepackage{booktabs,makecell,multirow}
\usepackage[capitalise,nosort]{cleveref}
\usepackage{subeqnarray,cases,color}
\crefname{equation}{}{}
\crefname{lemma}{Lemma}{Lemmas}
\crefname{theorem}{Theorem}{Theorems}
\crefname{discr}{Discretization}{Discretizations}
\crefname{subsection}{Subsection}{Subsections}

\apptocmd{\sloppy}{\hbadness 10000\relax}{}{}

\newcommand{\dual}[1]{\langle {#1} \rangle}

\newcommand{\nm}[1]{\lVert {#1} \rVert}
\newcommand{\Nm}[1]{\Big\| {#1} \Big\|}
\newcommand{\snm}[1]{\lvert {#1} \rvert}

\newcommand{\ssnm}[1]
{
  \left\vert\kern-0.25ex
  \left\vert\kern-0.25ex
  \left\vert
  {#1}
  \right\vert\kern-0.25ex
  \right\vert\kern-0.25ex
  \right\vert
}

\makeatletter
\def\spher@harm#1{%
  \vbox{\hbox{%
    \offinterlineskip
    \valign{&\hb@xt@2\p@{\hss$##$\hss}\vskip.2ex\cr#1\crcr}%
  }\vskip-.36ex}%
}
\def\gshone{\spher@harm{.}}
\def\gshtwo{\spher@harm{.&.}}
\def\gshthree{\spher@harm{.&.&.}}
\let\gsh\spher@harm
\makeatother

\newtheorem{lemma}{Lemma}[section]
\newtheorem{remark}{Remark}[section]
\newtheorem{theorem}{Theorem}[section]

\makeatletter\def\@captype{table}\makeatother

\begin{document}

\title{
  \Large\bf Discretization of a distributed optimal control problem with a stochastic
  parabolic equation driven by multiplicative noise
  \thanks{
    Binjie Li was supported in part by the National Natural
    Science Foundation of China (11901410).
  }
}
\author{Binjie Li\thanks{Corresponding author: libinjie@scu.edu.cn}}
\author{Qin Zhou\thanks{zqmath@aliyun.com}}
\affil{School of Mathematics, Sichuan University, Chengdu 610064, China}

\date{}
\maketitle

\begin{abstract} 
  A discretization of an optimal control problem of a stochastic parabolic
  equation driven by multiplicative noise is analyzed. The state equation is
  discretized by the continuous piecewise linear element method in space and by
  the backward Euler scheme in time. The convergence rate $ O(\tau^{1/2} + h^2)
  $ is rigorously derived.
\end{abstract}

\medskip\noindent{\bf Keywords:} optimal control, stochastic parabolic equation,
Brownian motion, discretization, convergence

\section{Introduction}
Let $ \{W(t) \mid\, t \geqslant 0\} $ be a one-dimensional Brownian motion
defined on a complete probability space $ (\Omega, \mathcal F, \mathbb P) $, and
let $ \mathbb F = \{\mathcal F_t\}_{t \geqslant 0} $ be the natural filtration
of $ \{W(t) \mid \, t \geqslant 0\} $. We use $ \mathbb EX$ to denote the
expectation of a scalar/vector-valued random variable $ X $ defined on $
(\Omega, \mathcal F, \mathbb P) $. Let $ \mathcal O \subset \mathbb R^d $ ($d=1,
2,3$) be a convex polytope, and let $ \Delta $ be the realization of the Laplace
operator with homogeneous Dirichlet boundary condition in $ L^2(\mathcal O) $.
Assume that $ 0 < \nu, T < \infty $. Let $ L_{\mathbb F}^2(\Omega;L^2(0,T;
L^2(\mathcal O))) $ be the set of all $ L^2(\mathcal O) $ valued $ \mathbb F
$-progressively measurable processes belonging to $
L^2(\Omega;L^2(0,T;L^2(\mathcal O))) $.
We consider the following stochastic optimal control problem:
\begin{small}
\begin{equation}
  \label{eq:model}
  \min_{
    \substack{
      u \in U_\text{ad} \\
      y \in L_{\mathbb F}^2(\Omega;L^2(0,T;L^2(\mathcal O)))
    }
  } J(u,y) := \frac12\mathbb E \Big(
    \nm{y-y_d}_{L^2(0,T;L^2(\mathcal O))}^2 +
    \nu \nm{u}_{L^2(0,T;L^2(\mathcal O))}^2
  \Big),
\end{equation}
\end{small}
subject to the state equation
\begin{equation}
  \label{eq:state}
  \begin{cases}
    \mathrm{d}y(t) =
    ( \Delta y+u )(t) \, \mathrm{d}t + y(t) \mathrm{d}W(t),
    \quad 0 \leqslant t \leqslant T, \\
    y(0) = 0,
  \end{cases}
\end{equation}
where $ y_d \in L_\mathbb F^2(\Omega; L^2(0,T; L^2(\mathcal O))) $ is given. The
above
admissible space $ U_\text{ad} $ is defined by
\begin{align*}
  U_\text{ad} &:= \big\{
    v \in L_\mathbb F^2(\Omega;L^2(0,T;L^2(\mathcal O))) \mid
    u_* \leqslant v \leqslant u^* \,\,
    \mathbb P \otimes \, \mathrm{d}t \otimes \, \mathrm{d}x
    \, \text{a.e.}
  \big\},
\end{align*}
where $ -\infty < u_* < u^* < \infty $ and $ \mathrm{d}t $ and $ \mathrm{d}x $
are respectively the Lebesgue measures in $ (0,T) $ and $ \mathcal O $.





There has been a vast amount of literature on the stochastic optimal control
theory, and by now it is still a very active research area. In this area, the
optimal control of the stochastic parabolic equations has been extensively
studied; see \cite{Bensoussan1983,Bensoussan1975,Du2013,Fuhrman2013,Hu1990,
Lu2014,Zhou1993} and the references therein. However, the numerical analysis of
these problems is quite rare. So far, in the literature we are aware of only one
paper addressing this issue. Dunst and Prohl~\cite{Dunst2016} analyzed a spatial
semi-discretization of an optimal control problem governed by a stochastic heat
equation with multiplicative noise. To our best knowledge, no convergence rate
is available for a full discretization of problem \cref{eq:model}.

The numerical analysis of problem \cref{eq:model} consists of two key
ingredients: the numerical analysis of a forward stochastic parabolic equation,
namely the state equation \cref{eq:state}; the numerical analysis of a backward
stochastic parabolic equation
\begin{equation} 
  \label{eq:co-state}
  \begin{cases}
    \mathrm{d}p(t) = -(\Delta p + y - y_d + z)(t) \, \mathrm{d}t +
    z(t) \, \mathrm{d}W(t), \quad 0 \leqslant t \leqslant T, \\
    p(T) = 0,
  \end{cases}
\end{equation}
namely the adjoint equation of problem \cref{eq:model}. There has already been a
considerable number of papers on the numerical analysis of the forward
stochastic parabolic equations; see \cite{Barbu2017, Barth2012, Cui_Hong_2019,
Jentzen2015, Kruse2014, Yan2005} and the references therein. There has also been
many papers on the numerical analysis of the backward stochastic differential
equations; see \cite{Bouchard2004, Crisan2012,
Hu_Dualart_2011,peng_xu_2011,Wang_Zhang_2011,Zhang2004,Zhao2006,Zhao2014,
Zhao2017} and the references therein. However, the numerical analysis of the
backward stochastic parabolic equations is rather limited.
Wang~\cite{WangY2016,Wang2020} analyzed a time-discretized Galerkin
approximation of a semilinear backward stochastic parabolic equation and a
time-discretization Galerkin approximation of a linear backward stochastic
parabolic equation. Recently, Li and Tang~\cite{Li_Tang_2021} developed a
splitting-up method for backward stochastic parabolic equations, where no
convergence rate was derived for the general nonlinear case.

In this paper, we consider the convergence of a discrete stochastic optimal
control problem. The state equation is discretized by the continuous piecewise
linear element method in space and the backward Euler scheme in time. The
stability and convergence of the discrete state equation can be easily derived
by the standard techniques. The main challenge in the numerical analysis is that
the process $ z $ in \cref{eq:co-state} is of low temporal regularity. Although
it appears that the numerical analysis in \cite{Wang2020} may be applied to
equation \cref{eq:co-state} under the condition that $ y_d \in L^2(0, T;
H_0^1(\mathcal O)) $ and
\[
  \nm{y_d(t) - y_d(s)}_{L^2(\mathcal O)} \leqslant
  C \snm{t-s}^{1/2} \quad
  \text{for all } 0 \leqslant s \leqslant t \leqslant T,
\]
throughout this paper we require only that $ y_d \in L_\mathbb
F^2(\Omega;L^2(0,T;L^2(\mathcal O))) $. To tackle the low temporal regularity of
the process $ z $, in the numerical analysis we propose a special discretization
of \cref{eq:co-state}. In this discretization the process $ z $ is only
discretized in space so that the convergence rate $ O(\tau^{1/2} + h^2) $ and
the stability of this discretization are derived. Finally, by the theoretical
results of this discretization, we are able to derive the following error
estimate:
\[
  \nm{y-Y}_{L^2(\Omega;L^2(0,T;L^2(\mathcal O)))} +
  \nm{u-U}_{L^2(\Omega;L^2(0,T;L^2(\mathcal O)))}
  \leqslant C(\tau^{1/2} + h^2),
\]
where $ U $ and $ Y $ are the numerical control and state, respectively.

The rest of this paper is organized as follows. \cref{sec:pre} introduces some
notations and the first order optimality condition of problem \cref{eq:model}.
\cref{sec:discr_pro} introduces a discrete stochastic optimal control problem
and presents its error estimate. Finally, \cref{sec:proofs} proves the error
estimate in the previous section.

\section{Preliminaries}
\label{sec:pre}
\subsection{Notations}
For two Banach spaces $ \mathcal B_1 $ and $ \mathcal B_2 $, $ \mathcal
L(\mathcal B_1, \mathcal B_2) $ denotes the set of all bounded linear operators
from $ \mathcal B_1 $ to $ \mathcal B_2 $; in particular, $ \mathcal L(\mathcal
B_1, \mathcal B_1) $ is abbreviated to $ \mathcal L(\mathcal B_1) $. The
identity mapping is denoted by $ I $. For any random variable $ v $ defined on $
(\Omega, \mathcal F, \mathbb P) $, $ \mathbb
E_t v $ means the expectation of $ v $ with respect to $ \mathcal F_t $ for each $ t
\geqslant 0 $. For any separable Banach space $ X $, define
\begin{small}
\begin{align*}
  L_{\mathbb F}^2(\Omega;L^2(0,T;X)) &:= \big\{
    \varphi: (0,T) \times \Omega\to X \mid\,
    \text{$ \varphi $ is $ \mathbb F $-progressively measurable} \\
    & \qquad\qquad\qquad\qquad\qquad\qquad \text{and } \int_0^T
    \mathbb E \nm{\varphi(t)}_X^2
    \, \mathrm{d}t < \infty
  \big\}.
\end{align*}
\end{small}

Recall that $ \Delta $ is the realization of the Laplace operator with
homogeneous Dirichlet boundary condition in $ L^2(\mathcal O) $. For any $ 0
\leqslant \beta < \infty $, define
\[ 
  \dot H^\beta(\mathcal O) := \left\{
    (-\Delta)^{-\beta/2}v: v \in L^2(\mathcal O)
  \right\}
\]
and endow this space with the norm
\[ 
  \nm{v}_{\dot H^\beta(\mathcal O)} :=
  \nm{(-\Delta)^{\beta/2} v}_{L^2(\mathcal O)}
  \quad \forall v \in \dot H^\beta(\mathcal O).
\]
We use $ \dot H^{-\beta}(\mathcal O) $ to denote the dual space of $ \dot
H^\beta(\mathcal O) $ for each $ \beta > 0 $. For any $ v \in L^2(\Omega;\dot
H^{-1}(\mathcal O)) $ and $ w \in L^2(\Omega;\dot H^1(\mathcal O)) $, define
\[ 
  [v,w] := \mathbb E \dual{v, w}_{\dot H^{-1}(\mathcal O), \dot H^1(\mathcal O)}
\]
where $ \dual{\cdot,\cdot}_{\dot H^{-1}(\mathcal O), \dot H^1(\mathcal O)} $
means the duality paring between $ \dot H^{-1}(\mathcal O) $ and $ \dot
H^1(\mathcal O) $. In particular, for any $ v, w \in L^2(\Omega;L^2(\mathcal O))
$,
\[
  [v,w] = \mathbb E \int_\mathcal O vw.
\]

\subsection{First order optimality condition}
For any $ g \in L_\mathbb F^2(\Omega;L^2(0,T;L^2(\mathcal O))) $, the following
forward stochastic parabolic equation
\begin{equation} 
  \label{eq:forward}
  \begin{cases}
    \mathrm{d}y(t) = (\Delta y+ g)(t) \, \mathrm{d}t +
    y(t) \, \mathrm{d}W(t) \quad \forall t \in [0,T] \\
    y(0) = 0
  \end{cases}
\end{equation}
admits a unique strong solution
\[ 
  y \in L_\mathbb F^2\big(
    \Omega; C([0,T];\dot H^1(\mathcal O))
    \cap L^2(0,T;\dot H^2(\mathcal O))
  \big).
\]
We summarize some standard results of $ y $ as follows: for any $ 0 \leqslant s
\leqslant
t \leqslant T $,
\begin{equation}
  \label{eq:S0-mild-form}
  \begin{aligned}
    & y(t) - e^{(t-s)\Delta}y(s) \\
    ={} &
    \int_{s}^{t} e^{(t-r)\Delta} g(r) \, \mathrm{d}r +
    \int_{s}^{t} e^{(t-r)\Delta} y(r) \, \mathrm{d}W(r),
  \end{aligned}
\end{equation}
and, in particular,
\begin{equation}
  \label{eq:S0-mild-form-2}
  y(t) = \int_0^t e^{(t-r)\Delta} g(r) \, \mathrm{d}r +
  \int_0^t e^{(t-r)\Delta}y(r) \, \mathrm{d}W(r);
\end{equation}
for any $ g \in L_\mathbb F^2(\Omega;L^2(0,T;\dot H^\beta(\mathcal O))) $ with $
\beta \geqslant 0 $,
\begin{equation}
  \label{eq:S0-regu}
  \begin{aligned}
    & \nm{y}_{
      L^2(\Omega;C([0,T];\dot H^{\beta+1}(\mathcal O)))
    } +
    \nm{y}_{
      L^2(\Omega;L^2(0,T;\dot H^{\beta+2}(\mathcal O)))
    } \\
    \leqslant{} &
    C \nm{g}_{
      L^2(\Omega;L^2(0,T;\dot H^\beta(\mathcal O)))
    },
  \end{aligned}
\end{equation}
where $ C $ is a positive constant depending only on $ T $. In the sequel, we
will use $ S_0g $ to denote the above $ y $.

For any $ g \in L_{\mathbb F}^2(\Omega;L^2(0,T;L^2(\mathcal O))) $, the backward
stochastic parabolic equation
\begin{equation}
  \label{eq:backward}
  \begin{cases}
    \mathrm{d}p(t) = -(\Delta p + g + z)(t) \, \mathrm{d}t +
    z(t) \, \mathrm{d}W(t) \quad \forall t \in [0,T] \\
    p(T) = 0
  \end{cases}
\end{equation}
admits a unique strong solution
\begin{align*} 
  (p,z) & \in L_{\mathbb F}^2\big(
    \Omega;C([0,T];\dot H^1(\mathcal O)) \cap L^2(0,T;\dot H^2(\mathcal O))
  \big) \\
  & \qquad {} \times L_\mathbb F^2(\Omega;L^2(0,T;\dot H^1(\mathcal O))).
\end{align*}
We summarize some standard results of $ (p,z) $ as follows: for any $ 0
\leqslant s \leqslant t \leqslant T $,
\begin{equation} 
  \label{eq:S1-S2-mild-form}
  \begin{aligned}
    & p(s) - e^{(t-s)\Delta}p(t) \\
    ={} &
    \int_s^t e^{(r-s)\Delta} (g+z)(r) \, \mathrm{d}r -
    \int_s^t e^{(r-s)\Delta} z(r) \, \mathrm{d}W(r);
  \end{aligned}
\end{equation}
for any $ g \in L_\mathbb F^2(\Omega;L^2(0,T;\dot H^\beta(\mathcal O))) $ with $
0 \leqslant \beta < \infty $,
\begin{equation} 
  \label{eq:S1-S2-regu}
  \begin{aligned}
    \nm{p}_{
      L^2(\Omega;C([0,T];\dot H^{\beta+1}(\mathcal O)))
    } + \nm{p}_{
      L^2(\Omega;L^2(0,T;\dot H^{\beta+2}(\mathcal O)))
    } \\
    {} + \nm{z}_{
      L^2(\Omega;L^2(0,T;\dot H^{\beta+1}(\mathcal O)))
    } \leqslant C \nm{g}_{
      L^2(\Omega;L^2(0,T;\dot H^\beta(\mathcal O)))
    },
  \end{aligned}
\end{equation}
where $ C $ is a positive constant depending only on $ T $. In the sequel, we
will use $ S_1g $ and $ S_2g $ to denote the above $ p $ and $ z $,
respectively.

\begin{remark} 
  The above properties of $ S_0 $ and $ S_1 $ are easily derived by the standard
  theory of stochastic differential equations (cf.~\cite[Chapters 2, 3 and
  5]{Pardoux2014}); see also \cite[Theorem 3.1]{Du2012}.
\end{remark}

By \cref{eq:forward,eq:backward}, applying the famous It\^o's formula yields
\[ 
  \int_0^T [S_0f, g] \, \mathrm{d}t =
  \int_0^T [f, S_1g] \, \mathrm{d}t
\]
for all $ f, g \in L_{\mathbb F}^2(\Omega;L^2(0,T;L^2(\mathcal O))) $. Hence, a
routine calculation gives the following first order optimality condition of
problem \cref{eq:model}.
\begin{theorem} 
  \label{thm:optim}
  Problem \cref{eq:model} admits a unique solution $ u \in U_\text{ad} $,
  and the following first order optimality condition holds:
  \begin{subequations}
  \begin{numcases}{}
    y = S_0 u, \label{eq:optim-a} \\
    p = S_1(y-y_d), \label{eq:optim-b} \\
    \int_0^T [p+\nu u, v-u] \, \mathrm{d}t
    \geqslant 0 \quad \text{for all $ v \in U_\text{ad} $.}
    \label{eq:optim-c}
  \end{numcases}
  \end{subequations}
\end{theorem}

\begin{remark} 
  For the proof of \cref{thm:optim}, we also refer the reader to
  \cite[Theorem~3.1]{Bensoussan1983} and \cite[Theorem~6.2]{Lu2016}.
\end{remark}

\section{Discretization}
\label{sec:discr_pro}
Let $ J > 0 $ be an integer and define $ \tau := T/J $. For each $ 0 \leqslant j
\leqslant J $, define $ t_j := j\tau $. We also set $ \delta W_j := W(t_{j+1}) -
W(t_j) $ for each $ 0 \leqslant j < J $. Let $ \mathcal K_h $ be a conventional
conforming, shape regular and quasi-uniform triangulation of $ \mathcal O $
consisting of $ d $-simplexes, and we use $ h $ to denote the maximum diameter
of the elements in $ \mathcal K_h $. Define
\begin{align*} 
  \mathcal V_h &:= \big\{
    v_h \in \dot H^1(\mathcal O) \mid\,
    v_h \text{ is piecewise linear on each }
    K \in \mathcal K_h
  \big\}, \\
  \mathcal X_{h,\tau} &:= \big\{
    V \in L_\mathbb F^2(\Omega;L^2(0,T;\mathcal V_h)) \mid\,
    \text{$V$ is constant on } [t_j,t_{j+1})
    \,\, \forall 0 \leqslant j < J
  \big\}.
\end{align*}
Let $ Q_h $ be the $ L^2(\mathcal O) $-orthogonal projection operator onto $
\mathcal V_h $, and define the discrete Laplace operator $ \Delta_h: \mathcal
V_h \to \mathcal V_h $ by
\[
  \int_{\mathcal O} (\Delta_h v_h) w_h =
  - \int_{\mathcal O} \nabla v_h \cdot \nabla w_h
\]
for all $ v_h, w_h \in \mathcal V_h $.

For any $ g \in L_\mathbb F^2(\Omega;L^2(0,T;L^2(\mathcal O))) $, define $
\mathcal S_0 g \in \mathcal
X_{h,\tau} $ by
\begin{small}
\begin{equation}
  \label{eq:calS0}
  \begin{cases}
    (\mathcal S_0 g)_{j\!+\!1} - (\mathcal S_0 g)_j =
    \tau \Delta_h (\mathcal S_0g)_{j\!+\!1} +
    \int_{t_j}^{t_{j\!+\!1}} \! Q_h g(t) \mathrm{d}t \!+\!
    (\mathcal S_0 g)_j \delta W_j, \, 0 \leqslant j < J, \\
    (\mathcal S_0g)_0 = 0.
  \end{cases}
\end{equation}
\end{small}
Here and in what follows, for any $ v \in L_\mathbb
F^2(\Omega;L^2(0,T;L^2(\mathcal O))) $ we use $ v_j $ to abbreviate $ v(t_j) $
for all $ 0 \leqslant j \leqslant J $.


The discretization of problem \cref{eq:model} is
\begin{equation}
  \label{eq:discretization}
  \min_{U \in U_\text{ad}^{h,\tau}}
  J_{h,\tau}(U) := \frac12 \mathbb E \big(
    \nm{\mathcal S_0 U - y_d}_{
      L^2(0,T;L^2(\mathcal O))
    }^2 +
    \nu\nm{U}_{L^2(0,T;L^2(\mathcal O))}^2
  \big),
\end{equation}
where
\[ 
  U_\text{ad}^{h,\tau} := \left\{
    V \in U_\text{ad} \mid
    \text{$V$ is constant on $ [t_j,t_{j+1}) $ for each $ 0 \leqslant j < J$}
  \right\}.
\]

The main result of this paper is the following error estimate.
\begin{theorem} 
  \label{thm:conv}
  Assume that $ y_d \in L_\mathbb F^2(\Omega;L^2(0,T;L^2(\mathcal O))) $. Let $
  u $ be the solution of problem \cref{eq:model}, and let $ U $ be the solution
  of problem \cref{eq:discretization}. Then
  \begin{equation}
    \label{eq:conv}
    \nm{S_0u \!-\! \mathcal S_0U}_{
      L^2(\Omega;L^2(0,T;L^2(\mathcal O)))
    } + \nm{u\!-\!U}_{
      L^2(\Omega;L^2(0,T;L^2(\mathcal O)))
    } \leqslant C(\tau^{1/2} + h^2),
  \end{equation}
  where $ C $ is a positive constant depending only on $ \nu $, $ u_* $, $ u^*
  $, $ y_d $, $ T $, $ \mathcal O $ and the regularity parameters of $ \mathcal
  V_h $.
\end{theorem}

\begin{remark} 
  Note that $ U_\text{ad}^{h,\tau} $ is not discretized in space, and this idea
  follows from \cite{Hinze2005}.
\end{remark}

\section{Proofs}
\label{sec:proofs}
For convenience, this section uses the following conventions: $ J > 2 $ and $
\tau < 2/5 $; $ a \lesssim b $ means that there exists a positive constant $ C
$, depending only on $ T $, $ \mathcal O $ and the regularity parameters of $
\mathcal V_h $, such that $ a \leqslant C b $. In addition, since they are
frequently used, we will use the following three properties implicitly: for any
$ f, g \in L_\mathbb F^2(\Omega;L^2(0,T,X)) $, where $ X $ is a Hilbert space
with inner product $ (\cdot,\cdot)_X $, it holds that (cf.~\cite[Theorem~7.1 and
Remark~7.1] {Baldi2017})
\begin{align*} 
  & \mathbb E \int_a^b f(t) \, \mathrm{d}W(t) =
  \mathbb E_a \int_a^b f(t) \, \mathrm{d}W(t) = 0, \\
  & \Nm{
    \int_a^b f(t) \, \mathrm{d}W(t)
  }_{L^2(\Omega;X)} = \nm{f}_{L^2(\Omega;L^2(a,b;X))}, \\
  & \mathbb E\Big(
    \int_a^b f(t) \, \mathrm{d}W(t), \,
    \int_a^b g(t) \, \mathrm{d}W(t)
  \Big)_X = \int_a^b \mathbb E \big( f(t), \, g(t) \big)_X \, \mathrm{d}t,
\end{align*}
where $ 0 \leqslant a < b \leqslant T $.

The main task of this section is to prove \cref{thm:conv}. We outline the
procedure as follows. In \cref{ssec:pre} we introduce some standard estimates.
In \cref{ssec:S0} we derive the stability and convergence of $ \mathcal S_0 $.
In \cref{ssec:sbde} we analyze a discretization of a backward stochastic
parabolic equation. In \cref{ssec:S1-S2} we introduce a discretization of the
adjoint equation of problem \cref{eq:model}, and, based on the theoretical
results in \cref{ssec:sbde}, we establish the stability and convergence of this
discretization. Finally, by the theoretical results in
\cref{ssec:S0,ssec:S1-S2}, we are able to conclude the proof of \cref{thm:conv}
in \cref{ssec:proof-conv}.

\subsection{Preliminary estimates}
\label{ssec:pre}
In this subsection, we summarize some standard estimates. For any $ \beta \in
\mathbb R $, let $ \dot H_h^\beta(\mathcal O) $ be the space of $ \mathcal V_h $
endowed with the norm
\[
  \nm{v_h}_{\dot H_h^\beta(\mathcal O)} :=
  \nm{(-\Delta_h)^{\beta/2} v_h}_{L^2(\mathcal O)}
  \quad \forall v_h \in \mathcal V_h.
\]
For any $ v_h \in \mathcal V_h $ and $ -1 \leqslant \beta \leqslant 1 $, we have
\begin{equation}
  \label{eq:H-Hh-equiv}
  \nm{v_h}_{\dot H^\beta(\mathcal O)} \lesssim
  \nm{v_h}_{\dot H_h^\beta(\mathcal O)} \lesssim
  \nm{v_h}_{\dot H^\beta(\mathcal O)}.
\end{equation}

\begin{lemma}
  \label{lem:e^tDelta}
  For any $ t > 0 $, we have
  \begin{align}
    \nm{e^{t\Delta}}_{\mathcal L(L^2(\mathcal O))}
    \leqslant 1, \label{eq:e^tDelta} \\
    \nm{I-e^{t\Delta}}_{
      \mathcal L(\dot H^2(\mathcal O), L^2(\mathcal O))
    }  \leqslant t, \label{eq:I-e^Delta-h2} \\
    \nm{I-e^{t\Delta}}_{
      \mathcal L(\dot H^1(\mathcal O), L^2(\mathcal O))
    }  \lesssim t^{1/2}.
    \label{eq:I-e^Delta-h1}
  \end{align}
\end{lemma}
\begin{proof}
  Inequality \cref{eq:e^tDelta} is standard. Since
  \[
    e^{t\Delta} - I = \Delta \int_0^t e^{s\Delta} \, \mathrm{d}s,
  \]
  we have
  \begin{align*}
    & \nm{e^{t\Delta} - I}_{
      \mathcal L(\dot H^2(\mathcal O, L^2(\mathcal O)))
    } = \Nm{
      \Delta \int_0^t e^{s\Delta} \, \mathrm{d}s
    }_{\mathcal L(\dot H^2(\mathcal O, L^2(\mathcal O)))} \\
    ={} &
    \Nm{
      \int_0^t e^{s\Delta} \, \mathrm{d}s \Delta
    }_{\mathcal L(\dot H^2(\mathcal O, L^2(\mathcal O)))}
    = \Nm{
      \int_0^t e^{s\Delta} \, \mathrm{d}s
    }_{\mathcal L(L^2(\mathcal O))} \\
    \leqslant{} &
    t \quad\text{(by \cref{eq:e^tDelta}),}
  \end{align*}
  which proves \cref{eq:I-e^Delta-h2}. It is clear, by \cref{eq:e^tDelta}, that
  \begin{equation}
    \label{eq:dddd}
    \Nm{I-e^{t\Delta}}_{
      \mathcal L(L^2(\mathcal O), L^2(\mathcal O))
    } \leqslant 2.
  \end{equation}
  Finally, in view of \cref{eq:I-e^Delta-h2,eq:dddd}, by interpolation
  (cf.~\cite[Theorems~2.6 and 4.36]{Lunardi2018}) we obtain \cref{eq:I-e^Delta-h1}. This
  completes the proof.
\end{proof}

\begin{lemma}
  \label{lem:e^tDeltah}
  For any $ t > 0 $, we have
  \begin{align}
    \nm{e^{t\Delta_h}Q_h}_{
      \mathcal L(\dot H_h^0(\mathcal O))
    } \leqslant 1, \label{eq:e^tDeltah} \\
    \nm{I-e^{t\Delta_h}}_{
      \mathcal L(\dot H_h^2(\mathcal O), \dot H_h^{0}(\mathcal O))
    } \leqslant t, \label{eq:I-e^Deltah-h2} \\
    \nm{I-e^{t\Delta_h}}_{
      \mathcal L(\dot H_h^1(\mathcal O), \dot H_h^{-1}(\mathcal O))
    } \leqslant t, \label{eq:I-e^Deltah-h1} \\
    \nm{I-e^{t\Delta_h}}_{
      \mathcal L(\dot H_h^0(\mathcal O), \dot H_h^{-1}(\mathcal O))
    } \lesssim t^{1/2}.
    \label{eq:I-e^Deltah-l2}
  \end{align}
\end{lemma}
\noindent Since the proof of this lemma is similar to that of
\cref{lem:e^tDelta}, it is omitted here.

\begin{lemma}
  \label{lem:lbj}
  Assume that $ g \in L^2(0,T;L^2(\mathcal O)) $. Let
  \begin{equation}
    \label{eq:2073}
    \eta(t) := \int_t^T e^{(s-t)\Delta_h}Q_hg(s) \, \mathrm{d}s,
    \quad 0 \leqslant t \leqslant T,
  \end{equation}
  and define $ \{P_j\}_{j=0}^J \subset \mathcal V_h $ by
  \begin{equation}
    \label{eq:2074}
    P_j - P_{j+1} = \tau \Delta_h P_j +
    \int_{t_j}^{t_{j+1}} Q_h g(t) \, \mathrm{d}t
  \end{equation}
  for all $ 0 \leqslant j < J $, where $ P_J := 0 $. Then
  \begin{align}
    \max_{0 \leqslant j \leqslant J}
    \Nm{ \eta(t_j) - P_j }_{\dot H_h^{-1}(\mathcal O)}
    \lesssim \tau
    \nm{g}_{L^2(0,T;L^2(\mathcal O))}, \label{eq:xx-1} \\
    \Big(
      \sum_{j=0}^{J-1} \nm{\eta - P_j}_{
        L^2(t_j,t_{j+1};L^2(\mathcal O))
      }^2
    \Big)^{1/2} \lesssim \tau
    \nm{g}_{L^2(0,T;L^2(\mathcal O))}.
    \label{eq:xx-2}
  \end{align}
\end{lemma}

\begin{lemma} 
  \label{lem:dg-spatial-g}
  For any $ g \in L^2(0,T;L^2(\mathcal O)) $,
  \begin{small}
  \begin{equation}
    \label{eq:0-1}
    \Big(
      \int_0^T \Nm{
        \int_t^T \big(
          e^{(s-t)\Delta} - e^{(s-t)\Delta_h} Q_h
        \big) g(s) \, \mathrm{d}s
      }_{L^2(\mathcal O)}^2 \, \mathrm{d}t
    \Big)^{1/2}
    \lesssim
    h^2 \nm{g}_{L^2(0,T;L^2(\mathcal O))}.
  \end{equation}
  \end{small}
\end{lemma}

\begin{lemma} 
  \label{lem:dg-conv-inf}
  If $ g \in L^2(0,T;L^2(\mathcal O)) $, then
  \begin{small}
  \begin{align*}
      & \sum_{j=0}^{J-1} \tau \Big\|
      \sum_{k=0}^j
      \int_{t_k}^{t_{k+1}}
      \big(
        e^{(t_{j+1} - t)\Delta} -
        (I-\tau\Delta_h)^{-(j-k+1)} Q_h
      \big) g(t) \, \mathrm{d}t
      \Big\|_{L^2(\mathcal O)}^2 \\
      \lesssim{} &
      (\tau + h^2)^2 \nm{g}_{L^2(0,T;L^2(\mathcal O))}^2.
  \end{align*}
  \end{small}
\end{lemma}

\begin{lemma} 
  \label{lem:dg-stab}
  Let $ g \in L^2(0,T;L^2(\mathcal O)) $. Define $ \{P_j\}_{j=0}^J \subset
  \mathcal V_h $ by
  \begin{equation}
    \begin{cases}
      P_j - P_{j+1} = \tau \Delta_h P_j +
      Q_h \int_{t_j}^{t_{j+1}} g(t) \, \mathrm{d}t,
      \quad 0 \leqslant j < J, \\
      P_J = 0.
    \end{cases}
  \end{equation}
  Then
  \begin{equation}
    \tau^{-1} \sum_{j=0}^{J-1} \nm{P_j - P_{j+1}}_{L^2(\mathcal O)}^2 +
    \tau \sum_{j=0}^{J-1} \nm{\Delta_h P_j}_{L^2(\mathcal O)}
    \lesssim \nm{g}_{L^2(0,T;L^2(\mathcal O))}^2.
  \end{equation}
\end{lemma}

\begin{lemma} 
  \label{lem:err_initial}
  For any $ 0 \leqslant j \leqslant J $,
  \[
    \Nm{
      e^{t_j\Delta} -
      (I-\tau\Delta_h)^{-j} Q_h
    }_{
      \mathcal L(\dot H^2(\mathcal O), L^2(\mathcal O))
    } \lesssim \tau + h^2.
  \]
\end{lemma}

\begin{remark} 
  The proof of \cref{lem:lbj} is presented in \cref{sec:lbj}.
  \cref{lem:dg-spatial-g} can be proved by the same argument as that used in the
  proof of \cite[Theorem~5.5]{Vexler2008I}. For the proof of
  \cref{lem:dg-conv-inf}, we refer the reader to
  \cite[Theorems~5.1 and 5.5]{Vexler2008I}. For the proof of
  \cref{lem:dg-stab}, we refer the reader to \cite[Corollary~4.7]{Vexler2008I}.
  For the proof of \cref{lem:err_initial}, we refer the reader to
  \cite[Theorems~3.1 and 7.1] {Thomee2006}.
\end{remark}

\subsection{Stability and convergence of \texorpdfstring{$ \mathcal S_0 $}{}}
\label{ssec:S0}
Let us first analyze the stability of $ \mathcal S_0 $.
\begin{lemma}
  \label{lem:S0-stab}
  For any $ g \in L_\mathbb F^2(\Omega;L^2(0,T;\dot H^{-1}(\mathcal O))) $, we
  have
  \begin{equation}
    \label{eq:S0-stab}
    \begin{aligned}
      & \max_{0 \leqslant j \leqslant J}
      \nm{(\mathcal S_0g)_j}_{L^2(\Omega;L^2(\mathcal O))} +
      \nm{\mathcal S_0g}_{
        L^2(\Omega;L^2(0,T;\dot H^1(\mathcal O)))
      } \\
      \lesssim{} &
      \nm{g}_{L^2(\Omega;L^2(0,T;\dot H^{-1}(\mathcal O)))}.
    \end{aligned}
  \end{equation}
\end{lemma}
\begin{proof}
  Let $ Y := \mathcal S_0g $. Fix $ 0 \leqslant j < J $. By \cref{eq:calS0} we
  have
  \begin{align*} 
    [Y_{j+1} - Y_j, \, Y_{j+1}] =
    \Big[
      \tau\Delta_h Y_{j+1} +
      \int_{t_j}^{t_{j+1}} g(t) \, \mathrm{d}t +
      Y_j \delta W_j, \, Y_{j+1}
    \Big],
  \end{align*}
  and so
  \begin{align*} 
    & \nm{Y_{j+1}}_{L^2(\Omega;L^2(\mathcal O))}^2 +
    \tau \nm{Y_{j+1}}_{L^2(\Omega;\dot H^1(\mathcal O))} \\
    ={} &
    [(1+\delta W_j)Y_j, \, Y_{j+1}] + \Big[
      \int_{t_j}^{t_{j+1}} g(t) \, \mathrm{d}t, \,
      Y_{j+1}
    \Big] \\
    \leqslant{} &
    \sqrt{1+\tau} \nm{Y_j}_{L^2(\Omega;L^2(\mathcal O))}
    \nm{Y_{j+1}}_{L^2(\Omega;L^2(\mathcal O))} + {} \\
    & \quad \sqrt\tau \nm{g}_{
      L^2(\Omega;L^2(t_j,t_{j+1};\dot H^{-1}(\mathcal O)))
    } \nm{Y_{j+1}}_{L^2(\Omega;\dot H^1(\mathcal O))} \\
    \leqslant{} &
    \frac12\nm{Y_j}_{L^2(\Omega;L^2(\mathcal O))}^2 +
    \frac{1+\tau}2 \nm{Y_{j+1}}_{L^2(\Omega;L^2(\mathcal O))}^2 + {} \\
    & \quad \nm{g}_{
      L^2(\Omega;L^2(t_j,t_{j+1};\dot H^{-1}(\mathcal O)))
    }^2 + \frac\tau4 \nm{Y_{j+1}}_{L^2(\Omega;\dot H^1(\mathcal O))}^2.
  \end{align*}
  It follows that
  \begin{equation}
    \label{eq:S0-stab-1}
    \begin{aligned}
      & \nm{Y_{j+1}}_{L^2(\Omega;L^2(\mathcal O))}^2 +
      \frac{3\tau}{2(1-\tau)} \nm{Y_{j+1}}_{
        L^2(\Omega;\dot H^1(\mathcal O))
      }^2 \\
      \leqslant{} &
      \frac1{1-\tau} \nm{Y_j}_{L^2(\Omega;L^2(\mathcal O))}^2 +
      \frac2{1-\tau} \nm{g}_{
        L^2(\Omega;L^2(t_j,t_{j+1};\dot H^{-1}(\mathcal O)))
      }^2.
    \end{aligned}
  \end{equation}
  By the fact $ Y_0 = 0 $, we then use the discrete Gronwall's inequality to
  derive
  \begin{equation}
    \label{eq:S0-stab-inf}
    \max_{0 \leqslant j \leqslant J}
    \nm{Y_j}_{L^2(\Omega;L^2(\mathcal O))}
    \lesssim \nm{g}_{
      L^2(\Omega;L^2(0,T;\dot H^{-1}(\mathcal O)))
    }.
  \end{equation}
  Moreover, summing both sides of \cref{eq:S0-stab-1} over $ j $ from $ 0 $ to $
  J-2 $ gives
  \begin{align*}
    & \sum_{j=0}^{J-2}
    \nm{Y_{j+1}}_{L^2(\Omega;L^2(\mathcal O))}^2 +
    \frac{3}{2(1-\tau)} \nm{Y}_{
      L^2(\Omega;L^2(0,T;\dot H^1(\mathcal O)))
    }^2 \\
    \leqslant{} &
    \frac1{1-\tau} \sum_{j=0}^{J-2} \nm{Y_j}_{
      L^2(\Omega;L^2(\mathcal O))
    }^2 + \frac2{1-\tau} \nm{g}_{
      L^2(\Omega;L^2(0,T;\dot H^{-1}(\mathcal O)))
    }^2,
  \end{align*}
  so that
  \begin{align*} 
    & \frac{3}{2(1-\tau)} \nm{Y}_{
      L^2(\Omega;L^2(0,T;\dot H^1(\mathcal O)))
    }^2 \\
    \leqslant{} &
    \Big(
      \frac1{1-\tau} - 1
    \Big) \sum_{j=0}^{J-2} \nm{Y_j}_{L^2(\Omega;
    L^2(\mathcal O))}^2 + \frac2{1-\tau} \nm{g}_{
      L^2(\Omega;L^2(0,T;\dot H^{-1}(\mathcal O)))
    }^2 \\
    ={} &
    \frac{\tau}{1-\tau} \sum_{j=0}^{J-2} \nm{Y_j}_{L^2(\Omega;
    L^2(\mathcal O))}^2 + \frac2{1-\tau} \nm{g}_{
      L^2(\Omega;L^2(0,T;\dot H^{-1}(\mathcal O)))
    }^2 \\
    \lesssim{}&
    \nm{g}_{L^2(\Omega;L^2(0,T;\dot H^{-1}(\mathcal O)))}^2
    \quad\text{(by \cref{eq:S0-stab-inf}).}
  \end{align*}
  It follows that
  \[ 
    \nm{Y}_{
      L^2(\Omega;L^2(0,T;\dot H^1(\mathcal O)))
    } \lesssim \nm{g}_{
      L^2(\Omega;L^2(0,T;\dot H^{-1}(\mathcal O)))
    }.
  \]
  Finally, combining \cref{eq:S0-stab-inf} and the above estimate proves
  \cref{eq:S0-stab} and hence this lemma.
\end{proof}

Then we analyze the convergence of $ \mathcal S_0 $, and the main result is the
following lemma.
\begin{lemma} 
  \label{thm:conv-S0}
  If $ g \in L_\mathbb F^2(\Omega;L^2(0,T;L^2(\mathcal O))) $, then
  \begin{equation}
    \label{eq:S0-conv}
    \nm{(S_0 - \mathcal S_0)g}_{
      L^2(\Omega;L^2(0,T;L^2(\mathcal O)))
    } \lesssim (\tau^{1/2} + h^2) \nm{g}_{
      L^2(\Omega;L^2(0,T;L^2(\mathcal O)))
    }.
  \end{equation}
\end{lemma}
\begin{remark}
  We refer the reader to \cite[Theorems 1.1 and 1.2]{Yan2005} for related
  results.
\end{remark}

To prove the above lemma, we first introduce the following three lemmas.

\begin{lemma}
  \label{lem:zq}
  Assume that $ 0 \leqslant k \leqslant j < J $. For any $ t_k \leqslant t
  \leqslant t_{k+1} $ we have
  \begin{equation}
    \label{eq:Dleta-Deltah}
    \Nm{
      e^{(t_{j+1}-t)\Delta} - (I-\tau \Delta_h)^{-(j-k+1)} Q_h
    }_{\mathcal L(\dot H^2(\mathcal O), L^2(\mathcal O))}
    \lesssim \tau + h^2.
  \end{equation}
\end{lemma}
\begin{proof} 
  Since
  \[ 
    e^{(t_{j+1} - t)\Delta} - e^{(t_{j+1}-t_k)\Delta}
    = e^{(t_{j+1} - t)\Delta} \big(
      I - e^{(t-t_k)\Delta}
    \big),
  \]
  by \cref{eq:e^tDelta,eq:I-e^Delta-h2} we obtain
  \[ 
    \Nm{
      e^{(t_{j+1}-t)\Delta} - e^{(t_{j+1}-t_k)\Delta}
    }_{\mathcal L(\dot H^2(\mathcal O), L^2(\mathcal O))}
    \leqslant \tau.
  \]
  Also, by \cref{lem:err_initial} we have
  \[ 
    \Nm{
      e^{(t_{j+1}-t_k)\Delta} -
      (I-\tau\Delta_h)^{-(j-k+1)} Q_h
    }_{
      \mathcal L(\dot H^2(\mathcal O), L^2(\mathcal O))
    } \lesssim \tau + h^2.
  \]
  Therefore, \cref{eq:Dleta-Deltah} follows from the estimate
  \begin{align*} 
    & \Nm{
      e^{(t_{j+1}-t)\Delta} - (I-\tau \Delta_h)^{-(j-k+1)} Q_h
    }_{\mathcal L(\dot H^2(\mathcal O), L^2(\mathcal O))} \\
    \leqslant{} &
    \Nm{
      e^{(t_{j+1}-t)\Delta} - e^{(t_{j+1}-t_k)\Delta}
    }_{\mathcal L(\dot H^2(\mathcal O), L^2(\mathcal O))} + {} \\
    & \quad \Nm{
      e^{(t_{j+1}-t_k)\Delta} - (I-\tau\Delta_h)^{-(j-k+1)} Q_h
    }_{\mathcal L(\dot H^2(\mathcal O), L^2(\mathcal O))}.
  \end{align*}
  This completes the proof.
\end{proof}

\begin{lemma}
  \label{lem:y-yj}
  For any $ g \in L_\mathbb F^2(\Omega;L^2(0,T;L^2(\mathcal O))) $, we have
  \begin{equation} 
    \label{eq:y-yj}
    \sum_{j=0}^{J-1} \nm{
      y - y(t_j)
    }_{L^2(\Omega;L^2(t_j,t_{j+1};L^2(\mathcal O)))}^2
    \lesssim
    \tau \nm{g}_{
      L^2(\Omega;L^2(0,T;L^2(\mathcal O)))
    }^2,
  \end{equation}
  where $ y:= S_0 g $.
\end{lemma}
\begin{proof} 
  For any $ t_j < t < t_{j+1} $ with $ 0 \leqslant j < J $, by
  \cref{eq:S0-mild-form} we have
  \[
    y(t) - e^{(t-t_j)\Delta}y(t_j) =
    \int_{t_j}^t e^{(s-t_j)\Delta} g(s) \, \mathrm{d}s +
    \int_{t_j}^t e^{(s-t_j)\Delta} y(s) \, \mathrm{d}W(s),
  \]
  and so 
  \begin{align*} 
    & \nm{
      y(t) - e^{(t-t_j)\Delta}y(t_j)
    }_{L^2(\Omega;L^2(\mathcal O))}^2 \\
    \leqslant{} &
    2\Nm{
      \int_{t_j}^t e^{(s-t_j)\Delta} g(s) \, \mathrm{d}s
    }_{L^2(\Omega;L^2(\mathcal O))}^2 +
    2\Nm{
      \int_{t_j}^t e^{(s-t_j)\Delta} y(s) \, \mathrm{d}W(s)
    }_{L^2(\Omega;L^2(\mathcal O))}^2 \\
    ={} &
    2\Nm{
      \int_{t_j}^t e^{(s-t_j)\Delta} g(s) \, \mathrm{d}s
    }_{L^2(\Omega;L^2(\mathcal O))}^2 +
    2 \int_{t_j}^t \nm{e^{(s-t_j)\Delta} y(s)}_{
      L^2(\Omega;L^2(\mathcal O))
    }^2 \, \mathrm{d}s \\
    \leqslant{} &
    2 \tau \int_{t_j}^t
    \nm{e^{(s-t_j)\Delta} g(s)}_{L^2(\Omega;L^2(\mathcal O))} ^2
    \, \mathrm{d}s +
    2 \int_{t_j}^t \nm{e^{(s-t_j)\Delta}y(s)}_{L^2(\Omega;L^2(\mathcal O))}^2
    \, \mathrm{d}s \\
    \leqslant{} &
    2 \tau \int_{t_j}^t
    \nm{g(s)}_{L^2(\Omega;L^2(\mathcal O))} ^2
    \, \mathrm{d}s +
    2 \int_{t_j}^t \nm{y(s)}_{L^2(\Omega;L^2(\mathcal O))}^2
    \, \mathrm{d}s \quad\text{(by \cref{eq:e^tDelta}).}
  \end{align*}
  It follows that
  \begin{align} 
    & \int_{t_j}^{t_{j+1}}
    \nm{y(t) - e^{(t-t_j)\Delta} y(t_j)}_{
      L^2(\Omega;L^2(\mathcal O))
    }^2 \, \mathrm{d}t \notag \\
    \leqslant{} &
    2 \int_{t_j}^{t_{j+1}}
    \int_{t_j}^t \tau \nm{g(s)}_{L^2(\Omega;L^2(\mathcal O))}^2 +
    \nm{y(s)}_{L^2(\Omega;L^2(\mathcal O))}^2 \, \mathrm{d}s
    \, \mathrm{d}t \notag \\
    ={} &
    2 \int_{t_j}^{t_{j+1}} (t_{j+1}-s) \big(
      \tau \nm{g(s)}_{L^2(\Omega;L^2(\mathcal O))} ^2 +
      \nm{y(s)}_{L^2(\Omega;L^2(\mathcal O))}^2
    \big) \, \mathrm{d}s \notag \\
    \leqslant{} &
    2\tau^2 \nm{g}_{
      L^2(\Omega;L^2(t_j,t_{j+1};L^2(\mathcal O)))
    }^2 + 2\tau \nm{y}_{
      L^2(\Omega;L^2(t_j,t_{j+1};L^2(\mathcal O)))
    }^2. \label{eq:y-yj-1}
  \end{align}
  We also have
  \begin{align}
    & \int_{t_j}^{t_{j+1}}
    \nm{(I-e^{(t-t_j)\Delta})y(t_j)}_{
      L^2(\Omega;L^2(\mathcal O))
    }^2 \, \mathrm{d}t \notag \\
    \lesssim{} &
    \tau \int_{t_j}^{t_{j+1}}
    \nm{y(t_j)}_{L^2(\Omega;\dot H^1(\mathcal O))}^2 \, \mathrm{d}t
    \quad\text{(by \cref{eq:I-e^Delta-h1})} \notag \\
    \lesssim{} &
    \tau^2 \nm{y}_{L^2(\Omega;C([0,T];\dot H^1(\mathcal O)))}^2.
    \label{eq:y-yj-2}
  \end{align}
  Combining \cref{eq:y-yj-1,eq:y-yj-2} yields
  \begin{align*}
    & \sum_{j=0}^{J-1} \nm{y-y(t_j)}_{
      L^2(\Omega;L^2(t_j,t_{j+1};L^2(\mathcal O)))
    }^2 \\
    \lesssim{} &
    \tau^2 \nm{g}_{L^2(\Omega;L^2(0,T;L^2(\mathcal O)))}^2 +
    \tau \nm{y}_{L^2(\Omega;L^2(0,T;L^2(\mathcal O)))}^2 \\
    & \quad {} + \tau \nm{y}_{
      L^2(\Omega;C([0,T];\dot H^1(\mathcal O)))
    }^2,
  \end{align*}
  so that \cref{eq:y-yj} follows from \cref{eq:S0-regu}. This completes the
  proof.
\end{proof}

\begin{lemma}
  \label{lem:wtY}
  Assume that $ g \in L_\mathbb F^2(\Omega;L^2(0,T;L^2(\mathcal O))) $. Define $
  \widetilde Y \in \mathcal X_{h,\tau} $ by
  \begin{small}
  \begin{equation} 
    \label{eq:wtYj-def}
    \begin{cases}
      \widetilde Y_{j+1} = \widetilde Y_j +
      \tau\Delta_h \widetilde Y_{j+1} + Q_h \Big(
        \int_{t_j}^{t_{j+1}}  g(t) \, \mathrm{d}t +
        \int_{t_j}^{t_{j+1}} y(t) \, \mathrm{d}W(t)
      \Big), \,\, 0 \leqslant j < J, \\
      \widetilde Y_0 = 0,
    \end{cases}
  \end{equation}
  \end{small}
  where $ y := S_0g $. Then
  \begin{equation} 
    \label{eq:y-wtY}
    \nm{y - \widetilde Y}_{
      L^2(\Omega;L^2(0,T;L^2(\mathcal O)))
    } \lesssim (\tau^{1/2}+h^2) \nm{g}_{
      L^2(\Omega;L^2(0,T;L^2(\mathcal O)))
    }.
  \end{equation}
\end{lemma}
\begin{proof} 
  Since
  \begin{align*} 
    & \nm{y - \widetilde Y}_{
      L^2(\Omega;L^2(0,T;L^2(\mathcal O)))
    }^2 \\
    \leqslant{} &
    2\sum_{j=0}^{J-1} \nm{y - y(t_j)}_{
      L^2(\Omega;L^2(t_j,t_{j+1}; L^2(\mathcal O)))
    }^2 + 2\tau \sum_{j=0}^{J-1}
    \nm{y(t_j) - \widetilde Y_j}_{L^2(\Omega;L^2(\mathcal O))}^2 \\
    \lesssim{} &
    \tau \nm{g}_{L^2(\Omega;L^2(0,T;L^2(\mathcal O)))}^2 +
    \tau \sum_{j=0}^{J-1} \nm{y(t_j) - \widetilde Y_j}_{
      L^2(\Omega;L^2(\mathcal O))
    }^2  \quad\text{(by \cref{lem:y-yj})},
  \end{align*}
  it suffices to prove
  \begin{equation} 
    \label{eq:yj-wtYj}
    \tau \sum_{j=0}^{J-1} \nm{y(t_j) - \widetilde Y_j}_{
      L^2(\Omega;L^2(\mathcal O))
    }^2 \lesssim (\tau + h^2)^2 \nm{g}_{
      L^2(\Omega;L^2(0,T;L^2(\mathcal O)))
    }^2.
  \end{equation}

  To this end, we proceed as follows. An inductive argument gives, for any $ 0
  \leqslant j < J $,
  \begin{equation} 
    \label{eq:wtYj}
    \widetilde Y_{j+1} = \sum_{k=0}^j (I-\tau \Delta_h)^{-(j-k+1)} Q_h
    \Big(
      \int_{t_k}^{t_{k+1}} g(t) \, \mathrm{d}t +
      \int_{t_k}^{t_{k+1}} y(t) \, \mathrm{d}W(t)
    \Big),
  \end{equation}
  and so from \cref{eq:S0-mild-form-2} we conclude that
  \begin{align} 
    & y(t_{j+1}) - \widetilde Y_{j+1} \notag \\
    ={} &
    \sum_{k=0}^j \int_{t_k}^{t_{k+1}}
    \big(
      e^{(t_{j+1}-t)\Delta} - (I-\tau\Delta_h)^{-(j-k+1)} Q_h
    \big) g(t) \, \mathrm{d}t \notag \\
    & \quad {} + \sum_{k=0}^j \int_{t_k}^{t_{k+1}} \big(
      e^{(t_{j+1}-t)\Delta} - (I-\tau\Delta_h)^{-(j-k+1)}
    \big) y(t) \, \mathrm{d}W(t) \notag \\
    =:{} &
    I_j + II_j \label{eq:S0-mild-formg-wtY}
  \end{align}
  for each $ 0 \leqslant j < J $. By \cref{lem:dg-conv-inf} we have
  \begin{equation} 
    \label{eq:II1}
    \tau \sum_{j=0}^{J-2} \nm{I_{j}}_{L^2(\Omega;L^2(\mathcal O))}^2
    \lesssim (\tau + h^2)^2 \nm{g}_{L^2(\Omega;L^2(0,T;L^2(\mathcal O)))}^2.
  \end{equation}
  For any $ 0 \leqslant j < J $, since
  \begin{align*} 
    & \Nm{
      \int_{t_k}^{t_{k+1}} \big(
        e^{(t_{j+1} - t)\Delta} -
        (I-\tau\Delta_h)^{-(j-k+1)} Q_h
      \big) y(t) \, \mathrm{d}W(t)
    }_{L^2(\Omega;L^2(\mathcal O))}^2 \\
    ={} &
    \int_{t_k}^{t_{k+1}} \big\|
      \big(
        e^{(t_{j+1}-t)\Delta} - (I-\tau\Delta_h)^{-(j-k+1)} Q_h
      \big) y(t)
    \big\|_{L^2(\Omega;L^2(\mathcal O))}^2 \, \mathrm{d}t \\
    \leqslant{} &
    \int_{t_k}^{t_{k+1}} \big\|
      \big(
        e^{(t_{j+1}-t)\Delta} - (I-\tau\Delta_h)^{-(j-k+1)} Q_h
      \big)
    \big\|_{\mathcal L(\dot H^2(\mathcal O), L^2(\mathcal O))}^2 \\
    & \qquad\qquad\qquad {} \times
    \nm{y(t)}_{L^2(\Omega;\dot H^2(\mathcal O))}^2 \, \mathrm{d}t \\
    \lesssim{} &
    (\tau + h^2)^2 \nm{ y(t) }_{
      L^2(\Omega;L^2(t_k,t_{k+1};\dot H^2(\mathcal O)))
    }^2 \quad\text{(by \cref{lem:zq}),}
  \end{align*}
  we get
  \begin{align*} 
    & \nm{II_{j}}_{L^2(\Omega;L^2(\mathcal O))}^2 \\
    ={} &
    \sum_{k=0}^j \Nm{
      \int_{t_k}^{t_{k+1}} \big(
        e^{(t_{j+1}-t)\Delta} - (I-\tau\Delta_h)^{-(j-k+1)} Q_h
      \big) y(t) \, \mathrm{d}W(t)
    }_{L^2(\Omega;L^2(\mathcal O))}^2 \\
    \lesssim{} &
    (\tau + h^2)^2 \nm{y}_{
      L^2(\Omega;L^2(0,t_{j+1};\dot H^2(\mathcal O)))
    }^2.
  \end{align*}
  It follows that
  \begin{align}
    \tau \sum_{j=0}^{J-2}
    \nm{II_{j}}_{L^2(\Omega;L^2(\mathcal O))}^2
    & \lesssim
    (\tau+h^2)^2 \nm{y}_{
      L^2(\Omega;L^2(0,T;\dot H^2(\mathcal O)))
    }^2 \notag \\
    & \lesssim
    (\tau+h^2)^2 \nm{g}_{
      L^2(\Omega;L^2(0,T;L^2(\mathcal O)))
    }^2 \quad\text{(by \cref{eq:S0-regu}).}
    \label{eq:II2}
  \end{align}
  Combining \cref{eq:S0-mild-formg-wtY,eq:II1,eq:II2} yields
  \begin{align*} 
    & \tau\sum_{j=0}^{J-2} \nm{
      y(t_{j+1}) - \widetilde Y_{j+1}
    }_{L^2(\Omega;L^2(\mathcal O))}^2 \\
    \lesssim{} &
    \tau \sum_{j=0}^{J-2} \big(
      \nm{I_{j}}_{L^2(\Omega;L^2(\mathcal O))}^2 +
      \nm{II_{j}}_{L^2(\Omega;L^2(\mathcal O))}^2
    \big) \\
    \lesssim{} &
    (\tau + h^2)^2 \nm{g}_{L^2(\Omega;L^2(0,T;L^2(\mathcal O)))}^2.
  \end{align*}
  Therefore, \cref{eq:yj-wtYj} follows from the fact that $ y(0) = \widetilde
  Y_0 = 0 $. This completes the proof.
\end{proof}

Finally, we are in a position to prove \cref{thm:conv-S0} as follows.

\medskip\noindent{\bf Proof of \cref{thm:conv-S0}}. Let
  \[
    y := S_0 g, \quad Y := \mathcal S_0 g, \quad
    E := Y - \widetilde Y,
  \]
  where $ \widetilde Y $ is defined by \cref{eq:wtYj-def}. From
  \cref{eq:calS0,eq:wtYj-def} we get
  \begin{align*} 
    \begin{cases}
      E_{j+1} = E_j + \tau \Delta_h E_{j+1} + E_j \delta W_j +
      \int_{t_j}^{t_{j+1}} (\widetilde Y- Q_h y)(t) \, \mathrm{d}W(t),
      \quad 0 \leqslant j < J, \\
      E_0 = 0,
    \end{cases}
  \end{align*}
  so that an induction argument gives
  \begin{equation}
    E_{j+1} = \sum_{k=0}^j (I-\tau \Delta_h)^{-(j-k+1)}
    (E_k \delta W_k + \eta_k)
  \end{equation}
  for all $ 0 \leqslant j < J $, where
  \[ 
    \eta_k := \int_{t_k}^{t_{k+1}}
    (\widetilde Y - Q_h y)(t) \, \mathrm{d}W(t),
    \quad 0 \leqslant k < J-1.
  \]
  It follows that, for any $ 0 \leqslant j < J $,
  \begin{small}
  \begin{align*}
    & \nm{E_{j+1}}_{L^2(\Omega;L^2(\mathcal O))}^2 \\
    \leqslant{} &
    2 \Nm{
      \sum_{k=0}^j (I-\tau\Delta_h)^{-(j-k+1)} E_k \delta W_k
    }_{L^2(\Omega;L^2(\mathcal O))}^2 +
    2 \Nm{
      \sum_{k=0}^j (I-\tau\Delta_h)^{-(j-k+1)} \eta_k
    }_{L^2(\Omega;L^2(\mathcal O))}^2 \\
    \leqslant{} &
    2 \tau \sum_{k=0}^j \nm{E_k}_{L^2(\Omega;L^2(\mathcal O))}^2 +
    2 \nm{y - \widetilde Y}_{L^2(\Omega;L^2(0,t_{j+1};L^2(\mathcal O)))}^2,
  \end{align*}
  \end{small}
  by the following two straightforward estimates:
  \begin{align*}
    & \Nm{
      \sum_{k=0}^j (I-\tau\Delta_h)^{-(j-k+1)}
      E_k \delta W_k
    }_{L^2(\Omega;L^2(\mathcal O))}^2 \\
    ={} &
    \sum_{k=0}^j \nm{
      (I-\tau\Delta_h)^{-(j-k+1)} E_k \delta W_k
    }_{L^2(\Omega;L^2(\mathcal O))}^2 \\
    ={} &
    \tau \sum_{k=0}^j \nm{
      (I-\tau\Delta_h)^{-(j-k+1)} E_k
    }_{L^2(\Omega;L^2(\mathcal O))}^2 \\
    \leqslant{} &
    \tau \sum_{k=0}^j \nm{ E_k }_{L^2(\Omega;L^2(\mathcal O))}^2
  \end{align*}
  and
  \begin{align*}
    & \Nm{
      \sum_{k=0}^j(I-\tau\Delta_h)^{-(j-k+1)}
      \eta_k
    }_{L^2(\Omega;L^2(\mathcal O))}^2 \\
    ={} &
    \sum_{k=0}^j \Nm{
      \int_{t_k}^{t_{k+1}} (I-\tau\Delta_h)^{-(j-k+1)}
      (\widetilde Y - Q_hy) \, \mathrm{d}W(t)
    }_{L^2(\Omega;L^2(\mathcal O))}^2 \\
    ={} &
    \sum_{k=0}^j \nm{
      (I-\tau\Delta_h)^{-(j-k+1)} (\widetilde Y - Q_h y)
    }_{L^2(\Omega;L^2(t_k,t_{k+1};L^2(\mathcal O)))}^2 \\
    \leqslant{} &
    \sum_{k=0}^j \nm{
      \widetilde Y - Q_h y
    }_{L^2(\Omega;L^2(t_k,t_{k+1};L^2(\mathcal O)))}^2 \\
    ={} &
    \nm{\widetilde Y - Q_hy}_{L^2(\Omega;L^2(0,t_{j+1};L^2(\mathcal O)))}^2 \\
    \leqslant{} &
    \nm{y-\widetilde Y}_{L^2(\Omega;L^2(0,t_{j+1};L^2(\mathcal O)))}^2.
  \end{align*}
  Hence, applying the discrete Gronwall's inequality yields, by the fact $ E_0 =
  0 $, that
  \[
    \max_{0 \leqslant j \leqslant J}
    \nm{E_j}_{L^2(\Omega;L^2(\mathcal O))}
    \lesssim \nm{y - \widetilde Y}_{
      L^2(\Omega;L^2(0,T;L^2(\mathcal O)))
    },
  \]
  which implies
  \[
    \nm{E}_{L^2(\Omega;L^2(0,T;L^2(\mathcal O)))}
    \lesssim \nm{y-\widetilde Y}_{
      L^2(\Omega;L^2(0,T;L^2(\mathcal O)))
    }.
  \]
  Therefore,
  \begin{align*} 
    \nm{y-Y}_{
      L^2(\Omega;L^2(0,T;L^2(\mathcal O)))
    } & \leqslant
    \nm{y-\widetilde Y}_{
      L^2(\Omega;L^2(0,T;L^2(\mathcal O)))
    } + \nm{E}_{
      L^2(\Omega;L^2(0,T;L^2(\mathcal O)))
    } \\
    & \lesssim
    \nm{y-\widetilde Y}_{
      L^2(\Omega;L^2(0,T;L^2(\mathcal O)))
    } \\
    & \lesssim
    (\tau^{1/2} + h^2) \nm{g}_{
      L^2(\Omega;L^2(0,T;L^2(\mathcal O)))
    } \quad\text{(by \cref{eq:y-wtY}).}
  \end{align*}
  This proves \cref{eq:S0-conv} and thus concludes the proof.
\hfill\ensuremath{\blacksquare}


\subsection{Discretization of a backward stochastic parabolic equation}
\label{ssec:sbde}
This subsection considers a discretization of the following backward stochastic
parabolic equation:
\begin{equation} 
  \label{eq:first_p}
  \begin{cases}
    \mathrm{d} p(t) =
    -( \Delta p + g )(t) \mathrm{d}t + z(t) \mathrm{d}W(t),
    \quad 0 \leqslant t \leqslant T, \\
    p(T) = 0,
  \end{cases}
\end{equation}
where $ g \in L_{\mathbb F}^2(\Omega;L^2(0,T;L^2(\mathcal O))) $ is given. We
summarize some standard results of the above equation as follows:
\begin{align}
  & p(t) = \int_t^T e^{(s-t)\Delta} g(s) \, \mathrm{d}s -
  \int_t^T e^{(s-t)\Delta} z(s) \, \mathrm{d}W(s), \quad
  0 \leqslant t \leqslant T,
  \label{eq:first_p_int} \\
  & p(t) = \int_t^T (\Delta p + g)(t) \, \mathrm{d}t -
  \int_t^T z(t) \, \mathrm{d}W(t),
  \quad 0 \leqslant t \leqslant T,
  \label{eq:first_p_strong} \\
  & \nm{p}_{
    L^2(\Omega;C([0,T];\dot H^1(\mathcal O)))
  } +
  \nm{p}_{
    L^2(\Omega;L^2(0,T;\dot H^2(\mathcal O)))
  } + {} \notag \\
  & \qquad \nm{z}_{
    L^2(\Omega;L^2(0,T;\dot H^1(\mathcal O)))
  } \leqslant C
  \nm{g}_{L^2(\Omega;L^2(0,T;L^2(\mathcal O)))},
  \label{eq:p-h2}
\end{align}
where $ C $ is a positive constant depending only on $ T $. The discretization
seeks $ P \in \mathcal X_{h,\tau} $ and $ Z \in L_\mathbb F^2(\Omega; L^2(0,
T;\mathcal V_h)) $ such that
\begin{equation}
  \label{eq:first-P}
  \begin{cases}
    P_j \!-\! P_{j+1} =
    \tau \Delta_h P_j + \int_{t_j}^{t_{j+1}} Q_h g(t) \,
    \mathrm{d}t - \int_{t_j}^{t_{j+1}} Z(t) \, \mathrm{d}W(t),
    \, 0 \leqslant j < J, \\
    P_J = 0.
  \end{cases}
\end{equation}
\begin{remark}
  The inspiration for discretization \cref{eq:first-P} is as follows. Since $ z
  $ in \cref{eq:first_p} is of low temporal regularity, for the moment we are
  unable to derive any appropriate convergence rate of $ Z $ if $ z $ is
  discretized in time; however, the convergence rate of $ Z $ is essential for
  our numerical analysis of the adjoint equation of problem \cref{eq:model}.
  Hence, in discretization \cref{eq:first-P}, $ z $ is only discretized in
  space.
\end{remark}

To analyze the convergence of discretization \cref{eq:first-P}, we first analyze
the convergence of the following spatial semidiscretization of equation
\cref{eq:first_p}:
\begin{equation} 
  \label{eq:first_ph}
  \begin{cases}
    \mathrm{d} p_h(t) = -(\Delta_h p_h + Q_h g)(t)
    \mathrm{d}t + z_h(t) \mathrm{d}W(t),
    \quad 0 \leqslant t \leqslant T, \\
    p_h(T) = 0.
  \end{cases}
\end{equation}
We summarize some standard properties of the above equation as follows. For any
$ \beta \in \mathbb R $, we have
\begin{equation} 
  \label{eq:ph-regu}
  \begin{aligned}
    & \nm{p_h}_{
      L^2(\Omega;C([0,T];\dot H_h^{\beta+1}(\mathcal O)))
    } +
    \nm{p_h}_{
      L^2(\Omega;L^2(0,T;\dot H_h^{\beta+2}(\mathcal O)))
    } \\
    & \quad {} + \nm{z_h}_{
      L^2(\Omega;L^2(0,T;\dot H_h^{\beta+1}(\mathcal O)))
    } \leqslant C
    \nm{Q_hg}_{L^2(\Omega;L^2(0,T;\dot H_h^\beta(\mathcal O)))},
  \end{aligned}
\end{equation}
where $ C $ is a positive constant depending only on $ T $.
For any $ 0 \leqslant s \leqslant t \leqslant T $,
\begin{equation} 
  \label{eq:ph-mild-form}
  p_h(s) - e^{(t-s)\Delta_h} p_h(t) =
  \int_s^{t} e^{(r-s)\Delta_h} Q_h g(r) \, \mathrm{d}r -
  \int_s^{t} e^{(r-s)\Delta_h} z_h(r) \, \mathrm{d}W(r),
\end{equation}
and, in particular,
\begin{equation} 
  \label{eq:first_ph_int}
  p_h(t) = \int_t^T e^{(r-t)\Delta_h} Q_h g(r) \, \mathrm{d}r -
  \int_t^T e^{(r-t)\Delta_h} z_h(r) \, \mathrm{d}W(r).
\end{equation}

\begin{remark}
  The proofs of \cref{eq:first_p_int,eq:p-h2,eq:ph-regu,eq:ph-mild-form}
  are similar to that of \cref{eq:S1-S2-mild-form,eq:S1-S2-regu}.
\end{remark}



\begin{lemma}
  \label{lem:p-ph}
  Assume that $ g \in L_\mathbb F^2(\Omega;L^2(0,T;L^2(\mathcal O))) $. Let $ (p,
  z) $ be the solution of equation \cref{eq:first_p}, and let $ (p_h,z_h) $ be
  the solution of equation \cref{eq:first_ph}. Then
  \begin{equation} 
    \label{eq:p-ph}
    \begin{aligned}
      & \nm{p-p_h}_{
        L^2(\Omega;C([0,T]; \dot H^{\beta+1}(\mathcal O)))
      } + \nm{p-p_h}_{
        L^2(\Omega;L^2(0,T;\dot H^{\beta+2}(\mathcal O)))
      } \\
      & \qquad {} + \nm{z-z_h}_{
        L^2(\Omega;L^2(0,T;\dot H^{\beta+1}(\mathcal O)))
      } \lesssim h^{-\beta} \nm{g}_{
        L^2(\Omega;L^2(0,T;L^2(\mathcal O)))
      }
    \end{aligned}
  \end{equation}
  for all $ -2 \leqslant \beta \leqslant -1 $.
\end{lemma}
\begin{proof} 
  Let
  \[
    e_h^p := p_h - Q_h p, \quad
    e_h^z := z_h - Q_h z.
  \]
  By \cref{eq:first_p_strong} we have
  \[ 
    Q_h p(t) = \int_t^T Q_h (\Delta p + g)(s) \, \mathrm{d}s -
    \int_t^T Q_h z(s) \, \mathrm{d}W(s),
    \quad 0 \leqslant t \leqslant T,
  \]
  so that
  \[
    \mathrm{d}(Q_h p)(t) = -Q_h(\Delta p + g)(t) \mathrm{d}t +
    Q_h z(t) \, \mathrm{d}W(t),
    \quad 0 \leqslant t \leqslant T.
  \]
  Hence, by \cref{eq:first_ph} we have
  \[
    \mathrm{d} e_h^p(t) = -\big(
      \Delta_h e_h^p + \Delta_h (Q_hp) - Q_h\Delta p
    \big)(t) \, \mathrm{d}t +
    e_h^z(t) \, \mathrm{d}W(t),
    \quad 0 \leqslant t \leqslant T.
  \]
  By the fact that $ e_h^p(T) = 0 $, we then deduce that $ (e_h^p, e_h^z) $ is
  the solution of equation \cref{eq:first_ph} with $ g $ replaced by $
  \Delta_h(Q_hp) - Q_h \Delta p $. It follows that
  \begin{small}
  \begin{align*} 
    & \nm{e_h^p}_{
      L^2(\Omega;C([0,T];\dot H_h^{\beta+1}(\mathcal O)))
    } + \nm{e_h^p}_{
      L^2(\Omega;L^2(0,T;\dot H_h^{\beta+2}(\mathcal O)))
    } + \nm{e_h^z}_{
      L^2(\Omega;L^2(0,T;\dot H_h^{\beta+1}(\mathcal O)))
    } \\
    \lesssim{} &
    \nm{\Delta_h(Q_h p) - Q_h\Delta p}_{
      L^2(\Omega;L^2(0,T;\dot H_h^{\beta}(\mathcal O)))
    } \quad\text{(by \cref{eq:ph-regu})}\\
    ={} &
    \nm{Q_hp - (\Delta_h)^{-1} Q_h\Delta p}_{
      L^2(\Omega;L^2(0,T;\dot H_h^{2+\beta}(\mathcal O)))
    } \\
    \lesssim{} &
    h^{-\beta} \nm{p}_{
      L^2(\Omega;L^2(0,T;\dot H^2(\mathcal O)))
    },
  \end{align*}
  \end{small}
  by the standard estimate
  \[
    \nm{Q_hv - (\Delta_h)^{-1} Q_h \Delta v}_{
      \dot H_h^{2+\beta}(\mathcal O)
    } \lesssim h^{-\beta} \nm{v}_{\dot H^2(\mathcal O)}
    \quad \forall v \in \dot H^2(\mathcal O).
  \]
  Hence, by \cref{eq:H-Hh-equiv} we get
  \begin{small}
  \begin{align*} 
    & \nm{e_h^p}_{
      L^2(\Omega;C([0,T];\dot H^{\beta+1}(\mathcal O)))
    } + \nm{e_h^p}_{
      L^2(\Omega;L^2(0,T;\dot H^{\beta+2}(\mathcal O)))
    } + \nm{e_h^z}_{
      L^2(\Omega;L^2(0,T;\dot H^{\beta+1}(\mathcal O)))
    } \\
    \lesssim{} &
    h^{-\beta} \nm{p}_{
      L^2(\Omega;L^2(0,T;\dot H^2(\mathcal O)))
    }.
  \end{align*}
  \end{small}
  On the other hand, we have the following standard estimate:
  \begin{align*} 
    & \nm{p-Q_hp}_{
      L^2(\Omega;C([0,T];\dot H^{\beta+1}(\mathcal O)))
    } + \nm{p-Q_hp}_{
      L^2(\Omega;L^2(0,T;\dot H^{\beta+2}(\mathcal O)))
    } \\
    & \qquad {} + \nm{z-Q_hz}_{
      L^2(\Omega;L^2(0,T;\dot H^{\beta+1}(\mathcal O)))
    } \\
    \lesssim{} &
    h^{-\beta} \nm{p}_{
      L^2(\Omega;C([0,T];\dot H^{1}(\mathcal O)))
    } + h^{-\beta} \nm{p}_{
      L^2(\Omega;L^2(0,T;\dot H^2(\mathcal O)))
    } \\
    & \qquad {} + h^{-\beta} \nm{z}_{
      L^2(\Omega;L^2(0,T;\dot H^1(\mathcal O)))
    }.
  \end{align*}
  Finally, combining the above two estimates yields
  \begin{align*} 
    & \nm{p-p_h}_{
      L^2(\Omega;C([0,T];\dot H^{\beta+1}(\mathcal O)))
    } + \nm{p-p_h}_{
      L^2(\Omega;L^2(0,T;\dot H^{\beta+2}(\mathcal O)))
    } \\
    & \qquad {} + \nm{z-z_h}_{
      L^2(\Omega;L^2(0,T;\dot H^{\beta+1}(\mathcal O)))
    } \\
    \lesssim{} &
    h^{-\beta} \nm{p}_{
      L^2(\Omega;C([0,T];\dot H^{1}(\mathcal O)))
    } + h^{-\beta} \nm{p}_{
      L^2(\Omega;L^2(0,T;\dot H^2(\mathcal O)))
    } \\
    & \qquad {} + h^{-\beta} \nm{z}_{
      L^2(\Omega;L^2(0,T;\dot H^1(\mathcal O)))
    } \\
    \lesssim{} &
    h^{-\beta} \nm{g}_{L^2(\Omega;L^2(0,T;L^2(\mathcal O)))}
    \quad\text{(by \cref{eq:p-h2}).}
  \end{align*}
  This proves \cref{eq:p-ph} and thus concludes the proof.
\end{proof}

\begin{lemma}
  \label{lem:ph-P}
  Assume that $ g \in L_{\mathbb F}^2(\Omega;L^2(0,T;L^2(\mathcal O))) $. Let $
  (p_h,z_h) $ be the solution of equation \cref{eq:first_ph}, and let $ (P, Z) $
  be the solution of \cref{eq:first-P}. Then
  \begin{align} 
    & \max_{0 \leqslant j \leqslant J}
    \nm{p_h(t_j) - P_j}_{
      L^2(\Omega;L^2(0,T;\dot H_h^{-1}(\mathcal O)))
    } \lesssim \tau \nm{g}_{L^2(\Omega;L^2(0,T;L^2(\mathcal O)))},
    \label{eq:ph-P-inf} \\
    & \nm{p_h - P}_{L^2(\Omega;L^2(0,T;L^2(\mathcal O)))}
    \lesssim \tau^{1/2} \nm{g}_{
      L^2(\Omega;L^2(0,T;L^2(\mathcal O)))
    }, \label{eq:ph-P-l2} \\
    & \nm{z_h-Z}_{
      L^2(\Omega;L^2(0,T;\dot H_h^{-1}(\mathcal O)))
    } \lesssim \tau^{1/2}
    \nm{g}_{L^2(\Omega;L^2(0,T;L^2(\mathcal O)))}.
    \label{eq:zh-Z}
  \end{align}
\end{lemma}
\begin{proof}
  Let us first prove \cref{eq:ph-P-inf}. By \cref{eq:first_ph_int} we have that
  \begin{equation} 
    \label{eq:p-eta}
    p_h(t) = \mathbb E_t \eta(t), \quad 0 \leqslant t \leqslant T,
  \end{equation}
  where
  \[
    \eta(t) := \int_t^T e^{(s-t)\Delta_h} Q_h g(s) \, \mathrm{d}s,
    \quad 0 \leqslant t \leqslant T.
  \]
  Define $ \widetilde P \in \mathcal X_{h,\tau} $ by
  \begin{equation}
    \label{eq:2}
    \begin{cases}
      \widetilde P_j - \widetilde P_{j+1} =
      \tau \Delta_h \widetilde P_j +
      Q_h \int_{t_j}^{t_{j+1}} g(t) \, \mathrm{d}t,
      \quad 0 \leqslant j < J, \\
      \widetilde P_J = 0.
    \end{cases}
  \end{equation}
  By definition, it is evident that
  \begin{equation} 
    \label{eq:P-wtP}
    P_j = \mathbb E_{t_j} \widetilde P_j
    \quad\text{for all $ 0 \leqslant j \leqslant J $}.
  \end{equation}
  By \cref{eq:p-eta,eq:P-wtP} we obtain
  \[
    p_h(t_{j})-P_j = \mathbb E_{t_j}(\eta(t_{j}) - \widetilde P_j)
    \quad\text{for all $ 0 \leqslant j \leqslant J $}.
  \]
  so that
  \begin{align*} 
    \max_{0 \leqslant j \leqslant J}
    \nm{p_h(t_j) - P_j}_{
      L^2(\Omega;\dot H_h^{-1}(\mathcal O))
    } & \leqslant
    \max_{0 \leqslant j \leqslant J}
    \nm{\eta(t_j) - \widetilde P_j}_{
      L^2(\Omega;\dot H_h^{-1}(\mathcal O))
    }.
  \end{align*}
  Hence, the desired estimate \cref{eq:ph-P-inf} follows from \cref{eq:xx-1}.

  Then let us prove \cref{eq:ph-P-l2}. Let $ 0 \leqslant j < J $ be arbitrary
  but fixed. For any $ t_j < t \leqslant t_{j+1} $, since \cref{eq:ph-mild-form}
  implies
  \[
    p_h(t_j) - e^{(t-t_j)\Delta_h}p_h(t) =
    \int_{t_j}^t e^{(s-t_j)\Delta_h} Q_h g(s) \, \mathrm{d}s -
    \int_{t_j}^t e^{(s-t_j)\Delta_h} z_h(s) \, \mathrm{d}W(s),
  \]
  we have
  \begin{small}
  \begin{align*} 
    & \nm{
      p_h(t_j) - e^{(t-t_j)\Delta_h}p_h(t)
    }_{L^2(\Omega;L^2(\mathcal O))} \\
    \leqslant{} &
    \Big\|
    \int_{t_j}^t e^{(s-t_j)\Delta_h} Q_h g(s) \, \mathrm{d}s
    \Big\|_{L^2(\Omega;L^2(\mathcal O))} +
    \Big\|
    \int_{t_j}^t e^{(s-t_j)\Delta_h} z_h(s) \, \mathrm{d}W(s)
    \Big\|_{L^2(\Omega;L^2(\mathcal O))} \\
    \leqslant{} &
    \int_{t_j}^t \nm{
      e^{(s-t_j)\Delta_h} Q_h g(s)
    }_{L^2(\Omega;L^2(\mathcal O))}
    \, \mathrm{d}s +
    \Big(
      \int_{t_j}^t \nm{
        e^{(s-t_j)\Delta_h} z_h(s)
      }_{L^2(\Omega;L^2(\mathcal O))}^2 \, \mathrm{d}s
    \Big)^{1/2}  \\
    \leqslant{} &
    \int_{t_j}^t \nm{ Q_h g(s) }_{L^2(\Omega;L^2(\mathcal O))}
    \, \mathrm{d}s +
    \Big(
      \int_{t_j}^t \nm{ z_h(s) }_{
        L^2(\Omega;L^2(\mathcal O))
      }^2 \, \mathrm{d}s
    \Big)^{1/2} \quad\text{(by \cref{eq:e^tDeltah})} \\
    \leqslant{} &
    \int_{t_j}^t \nm{g(s)}_{L^2(\Omega;L^2(\mathcal O))}
    \, \mathrm{d}s +
    \Big(
      \int_{t_j}^t \nm{ z_h(s) }_{
        L^2(\Omega;L^2(\mathcal O))
      }^2 \, \mathrm{d}s
    \Big)^{1/2} \\
    \leqslant{} &
    \sqrt{t-t_j} \Big(
      \int_{t_j}^t \nm{g(s)}_{L^2(\Omega;L^2(\mathcal O))}^2
      \, \mathrm{d}s
    \Big)^{1/2} + \Big(
      \int_{t_j}^t \nm{ z_h(s) }_{
        L^2(\Omega;L^2(\mathcal O))
      }^2 \, \mathrm{d}s
    \Big)^{1/2}.
  \end{align*}
  \end{small}
  It follows that
  \begin{align*} 
    & \int_{t_j}^{t_{j+1}}
    \nm{
      p_h(t_j) - e^{(t-t_j)\Delta_h}p_h(t)
    }_{
      L^2(\Omega;L^2(\mathcal O))
    }^2 \, \mathrm{d}t \\
    \leqslant{} &
    2 \int_{t_j}^{t_{j+1}} \Big(
      (t-t_j) \int_{t_j}^t
      \nm{g(s)}_{L^2(\Omega;L^2(\mathcal O))}^2 \, \mathrm{d}s +
      \int_{t_j}^t \nm{z_h(s)}_{L^2(\Omega;L^2(\mathcal O))}^2
      \, \mathrm{d}s \,
    \Big) \mathrm{d}t \\
    \leqslant{} &
    \tau^2 \nm{g}_{
      L^2(\Omega;L^2(t_j,t_{j+1};L^2(\mathcal O)))
    }^2 + 2\tau \nm{z_h}_{
      L^2(\Omega;L^2(t_j,t_{j+1};L^2(\mathcal O)))
    }^2.
  \end{align*}
  In addition, by \cref{eq:I-e^Deltah-h2} we have
  \[ 
    \int_{t_j}^{t_{j+1}} \nm{
      (I-e^{(t-t_j)\Delta_h}) p_h(t)
    }_{L^2(\Omega;L^2(\mathcal O))}^2 \, \mathrm{d}t
    \leqslant \tau^2 \int_{t_j}^{t_{j+1}}
    \nm{p_h(t)}_{
      L^2(\Omega;\dot H_h^2(\mathcal O))
    }^2 \, \mathrm{d}t.
  \]
  Combining the above two estimates gives
  \begin{align*} 
    & \sum_{j=0}^{J-1} \int_{t_j}^{t_{j+1}}
    \nm{p_h(t_j) - p_h(t)}_{L^2(\Omega;L^2(\mathcal O))}^2
    \, \mathrm{d}t \\
    \lesssim{} &
    \tau^2 \nm{g}_{
      L^2(\Omega;L^2(0,T;L^2(\mathcal O)))
    }^2 + \tau \nm{z_h}_{
      L^2(\Omega;L^2(t_j,t_{j+1};L^2(\mathcal O)))
    }^2 \\
    & \quad {} + \tau^2 \nm{p_h}_{
      L^2(\Omega;L^2(0,T;\dot H_h^2(\mathcal O)))
    }^2 \\
    \lesssim{} &
    \tau \nm{g}_{L^2(\Omega;L^2(0,T;L^2(\mathcal O)))}^2
    \quad\text{(by \cref{eq:ph-regu}).}
  \end{align*}
  It follows that
  \begin{align} 
    & \sum_{j=0}^{J-1} \int_{t_j}^{t_{j+1}}
    \nm{(I-\mathbb E_{t_j})p_h(t)}_{
      L^2(\Omega;L^2(\mathcal O))
    }^2 \, \mathrm{d}t \notag \\
    \leqslant{} &
    \sum_{j=0}^{J-1} \int_{t_j}^{t_{j+1}}
    \nm{p_h(t)-p_h(t_j)}_{
      L^2(\Omega;L^2(\mathcal O))
    }^2 \, \mathrm{d}t \notag \\
    \lesssim{} &
    \tau \nm{g}_{L^2(\Omega;L^2(0,T;L^2(\mathcal O)))}^2.
    \label{eq:jm-1}
  \end{align}
  We also have
  \begin{align} 
    & \sum_{j=0}^{J-1} \int_{t_j}^{t_{j+1}}
    \nm{
      \mathbb E_{t_j}(p_h(t) - \widetilde P_j)
    }_{L^2(\Omega;L^2(\mathcal O))}^2 \notag \\
    ={}&
    \sum_{j=0}^{J-1} \int_{t_j}^{t_{j+1}}
    \nm{
      \mathbb E_{t_j}(\eta(t) - \widetilde P_j)
    }_{L^2(\Omega;L^2(\mathcal O))}^2
    \quad\text{(by \cref{eq:p-eta})} \notag \\
    \leqslant{} &
    \sum_{j=0}^{J-1} \int_{t_j}^{t_{j+1}}
    \nm{
      \eta(t) - \widetilde P_j
    }_{L^2(\Omega;L^2(\mathcal O))}^2 \notag \\
    \lesssim{} &
    \tau^2 \nm{g}_{
      L^2(\Omega;L^2(0,T;L^2(\mathcal O)))
    }^2 \quad\text{(by \cref{eq:xx-2}).}
    \label{eq:jm-2}
  \end{align}
  A straightforward computation gives
  \begin{align*} 
    & \nm{p_h - P}_{L^2(\Omega;L^2(0,T;L^2(\mathcal O)))}^2 \\
    ={} &
    \sum_{j=0}^{J-1} \int_{t_j}^{t_{j+1}}
    \nm{p_h(t) - P_j}_{L^2(\Omega;L^2(\mathcal O))}^2 \, \mathrm{d}t \\
    ={} &
    \sum_{j=0}^{J-1} \int_{t_j}^{t_{j+1}}
    \nm{
      p_h(t) - \mathbb E_{t_j} \widetilde P_j
    }_{L^2(\Omega;L^2(\mathcal O))}^2 \, \mathrm{d}t
    \quad\text{(by \cref{eq:P-wtP})} \\
    ={} &
    \sum_{j=0}^{J-1} \int_{t_j}^{t_{j+1}}
    \nm{
      (I-\mathbb E_{t_j})p_h(t) +
      \mathbb E_{t_j}(p_h(t) - \widetilde P_j)
    }_{L^2(\Omega;L^2(\mathcal O))}^2 \, \mathrm{d}t \\
    \leqslant{} &
    2\sum_{j=0}^{J-1} \int_{t_j}^{t_{j+1}}
    \nm{
      (I-\mathbb E_{t_j})p_h(t)
    }_{
      L^2(\Omega;L^2(\mathcal O))
    }^2 + \nm{
      \mathbb E_{t_j}(p_h(t) - \widetilde P_j)
    }_{
      L^2(\Omega;L^2(\mathcal O))
    }^2 \, \mathrm{d}t \\
    \lesssim{} &
    \tau \nm{g}_{L^2(\Omega;L^2(0,T;L^2(\mathcal O)))}^2
    \quad\text{(by \cref{eq:jm-1,eq:jm-2}).}
  \end{align*}
  This proves estimate \cref{eq:ph-P-l2}.

  Finally, we prove estimate \cref{eq:zh-Z}. Let $ 0 \leqslant j < J $ be
  arbitrary but fixed. By \cref{eq:ph-mild-form} we have
  \[ 
    p_h(t_j) - e^{\tau\Delta_h} p_h(t_{j+1}) =
    \int_{t_j}^{t_{j+1}} e^{(t-t_j)\Delta_h} Q_h g(t) \, \mathrm{d}t -
    \int_{t_j}^{t_{j+1}} e^{(t-t_j)\Delta_h} z_h(t) \, \mathrm{d}W(t),
  \]
  and so
  \begin{align*} 
    & \int_{t_j}^{t_{j+1}}
    e^{(t-t_j)\Delta_h} z_h(t) \mathrm{d}W(t) =
    (I-\mathbb E_{t_j}) \int_{t_j}^{t_{j+1}}
    e^{(t-t_j)\Delta_h} z_h(t) \mathrm{d}W(t) \\
    ={} &
    (I - \mathbb E_{t_j})\Big(
      e^{\tau\Delta_h} p_h(t_{j+1}) - p_h(t_j) +
      \int_{t_j}^{t_{j+1}}
      e^{(t-t_j)\Delta_h} Q_h g(t) \mathrm{d}t
    \Big) \\
    ={} &
    (I - \mathbb E_{t_j})\Big(
      e^{\tau\Delta_h} p_h(t_{j+1}) +
      \int_{t_j}^{t_{j+1}}
      e^{(t-t_j)\Delta_h} Q_h g(t) \mathrm{d}t
    \Big),
  \end{align*}
  by the fact that $ p_h(t_j) $ is $ \mathcal F_{t_j} $-measurable.
  It follows that
  \begin{align*} 
    & \int_{t_j}^{t_{j+1}} z_h(t) \, \mathrm{d}W(t) \\
    ={} &
    \int_{t_j}^{t_{j+1}} (I-e^{(t-t_j)\Delta_h}) z_h(t) \, \mathrm{d}W(t) +
    \int_{t_j}^{t_{j+1}} e^{(t-t_j)\Delta_h} z_h(t) \, \mathrm{d}W(t) \\
    ={} &
    \int_{t_j}^{t_{j+1}} (I-e^{(t-t_j)\Delta_h}) z_h(t) \, \mathrm{d}W(t) + {} \\
    & \quad (I-\mathbb E_{t_j}) \Big(
      e^{\tau\Delta_h}p_h(t_{j+1}) +
      \int_{t_j}^{t_{j+1}} e^{(t-t_j)\Delta_h} Q_h g(t) \, \mathrm{d}t
    \Big).
  \end{align*}
  Therefore, since \cref{eq:first-P} implies
  \begin{align*}
    \int_{t_j}^{t_{j+1}} Z(t) \, \mathrm{d}W(t) =
    (I-\mathbb E_{t_j}) \Big(
      P_{j+1} + \int_{t_j}^{t_{j+1}} Q_h g(t) \, \mathrm{d}t
    \Big),
  \end{align*}
  we obtain
  \begin{equation} 
    \label{eq:zh-Z-I1-I2-I3}
    \int_{t_j}^{t_{j+1}} (z_h - Z)(t) \, \mathrm{d}W(t) =
    \mathbb I_1 + \mathbb I_2 + \mathbb I_3,
  \end{equation}
  where
  \begin{align*} 
    \mathbb I_1 &:= \int_{t_j}^{t_{j+1}}
    \big( I - e^{(t-t_j)\Delta_h} \big) z_h(t) \, \mathrm{d}W(t), \\
    \mathbb I_2 &:= (I - \mathbb E_{t_j}) \big(
      e^{\tau\Delta_h}p_h(t_{j+1}) - P_{j+1}
    \big), \\
    \mathbb I_3 &:= (I-\mathbb E_{t_j}) \int_{t_j}^{t_{j+1}}
    \big( e^{(t-t_j)\Delta_h} - I \big) Q_h g(t) \, \mathrm{d}t.
  \end{align*}
  For $ \mathbb I_1 $ we have
  \begin{align*} 
    \nm{\mathbb I_1}_{L^2(\Omega;\dot H_h^{-1}(\mathcal O))}
    &= \Big(
      \int_{t_j}^{t_{j+1}} \Nm{
        (I-e^{(t-t_j)\Delta_h}) z_h(t)
      }_{L^2(\Omega;\dot H_h^{-1}(\mathcal O))}^2
      \, \mathrm{d}t
    \Big)^{1/2} \\
    &\leqslant
    \tau \nm{z_h}_{
      L^2(\Omega;L^2(t_j,t_{j+1};\dot H_h^{1}(\mathcal O)))
    } \quad\text{(by \cref{eq:I-e^Deltah-h1}).}
  \end{align*}
  For $ \mathbb I_2 $ we have
  \begin{align*} 
    & \nm{\mathbb I_2}_{L^2(\Omega;\dot H_h^{-1}(\mathcal O))}
    \leqslant \nm{
      e^{\tau\Delta_h}p_h(t_{j+1}) - P_{j+1}
    }_{L^2(\Omega;\dot H_h^{-1}(\mathcal O))} \\
    \leqslant{} &
    \nm{
      (I-e^{\tau\Delta_h})p_h(t_{j+1})
    }_{L^2(\Omega;\dot H_h^{-1}(\mathcal O))} +
    \nm{p_h(t_{j+1}) - P_{j+1}}_{L^2(\Omega;\dot H_h^{-1}(\mathcal O))} \\
    \lesssim{} &
    \tau \nm{p_h(t_{j+1})}_{L^2(\Omega;\dot H_h^1(\mathcal O))} +
    \tau \nm{g}_{L^2(\Omega;L^2(0,T;L^2(\mathcal O)))}
    \quad\text{(by \cref{eq:I-e^Deltah-h1,eq:ph-P-inf})} \\
    \lesssim{} &
    \tau \nm{g}_{L^2(\Omega;L^2(0,T;L^2(\mathcal O)))}
    \quad\text{(by \cref{eq:ph-regu}).}
  \end{align*}
  For $ \mathbb I_3 $ we have
  \begin{align*} 
    \nm{\mathbb I_3}_{L^2(\Omega;\dot H_h^{-1}\mathcal O))}
    & \leqslant
    \int_{t_j}^{t_{j+1}} \Nm{
      (e^{(t-t_j)\Delta_h} - I) Q_hg(t)
    }_{L^2(\Omega;\dot H_h^{-1}(\mathcal O))} \, \mathrm{d}t \\
    & \lesssim  \tau^{1/2} \int_{t_j}^{t_{j+1}}
    \nm{Q_hg(t)}_{L^2(\Omega;L^2(\mathcal O))} \, \mathrm{d}t
    \quad\text{(by \cref{eq:I-e^Deltah-l2})} \\
    & \lesssim \tau \nm{Q_hg}_{
      L^2(\Omega;L^2(t_j,t_{j+1};L^2(\mathcal O)))
    } \\
    & \leqslant \tau \nm{g}_{
      L^2(\Omega;L^2(t_j,t_{j+1};L^2(\mathcal O)))
    }.
  \end{align*}
  Combining \cref{eq:zh-Z-I1-I2-I3} and the above estimates of $ \mathbb I_1 $,
  $ \mathbb I_2 $ and $ \mathbb I_3 $ yields
  \begin{align*} 
    & \nm{z_h-Z}_{
      L^2(\Omega;L^2(t_j,t_{j+1};\dot H_h^{-1}(\mathcal O)))
    } \\
    ={} &
    \Nm{
      \int_{t_j}^{t_{j+1}} (z_h-Z)(t) \, \mathrm{d}W(t)
    }_{L^2(\Omega;\dot H_h^{-1}(\mathcal O))} \\
    \lesssim{} &
    \tau\nm{z_h}_{
      L^2(\Omega;L^2(t_j,t_{j+1};\dot H_h^1(\mathcal O)))
    } + \tau \nm{g}_{
      L^2(\Omega;L^2(0,T;L^2(\mathcal O)))
    },
  \end{align*}
  and hence
  \begin{align*} 
    & \nm{z_h-Z}_{
      L^2(\Omega;L^2(0,T;\dot H_h^{-1}(\mathcal O)))
    } \\
    \lesssim{} &
    \tau \nm{z_h}_{
      L^2(\Omega;L^2(0,T;\dot H_h^1(\mathcal O)))
    } + \tau^{1/2} \nm{g}_{
      L^2(\Omega;L^2(0,T;L^2(\mathcal O)))
    } \\
    \lesssim{} &
    \tau^{1/2} \nm{g}_{
      L^2(\Omega;L^2(0,T;L^2(\mathcal O)))
    } \quad\text{(by \cref{eq:ph-regu}).}
  \end{align*}
  This proves \cref{eq:zh-Z} and thus completes the proof.
\end{proof}
\begin{remark} 
  \label{rem:Z-stab}
  Under the condition of \cref{lem:ph-P}, a simple modification of the proof of
  \cref{eq:zh-Z} yields
  \[
    \nm{z_h-Z}_{
      L^2(\Omega;L^2(0,T;L^2(\mathcal O)))
    } \lesssim \nm{g}_{
      L^2(\Omega;L^2(0,T;L^2(\mathcal O)))
    },
  \]
  so that by \cref{eq:ph-regu} we obtain
  \[
    \nm{Z}_{
      L^2(\Omega;L^2(0,T;L^2(\mathcal O)))
    } \lesssim \nm{g}_{
      L^2(\Omega;L^2(0,T;L^2(\mathcal O)))
    }.
  \]
\end{remark}

Finally, from \cref{lem:p-ph,lem:ph-P} and estimate \cref{eq:H-Hh-equiv} we
conclude the following error estimate.
\begin{lemma} 
  \label{lem:first-p-P}
  Assume that $ g \in L_{\mathbb F}^2(\Omega;L^2(0,T;L^2(\mathcal O))) $. Let $
  (p,z) $ be the solution of \cref{eq:first_p} and let $ (P,Z) $ be the solution
  of \cref{eq:first-P}. Then
  \begin{equation}
    \label{eq:first-p-P}
    \begin{aligned}
      & \nm{p-P}_{L^2(\Omega;L^2(0,T;L^2(\mathcal O)))} +
      \nm{z-Z}_{L^2(\Omega;L^2(0,T;\dot H^{-1}(\mathcal O)))} \\
      \lesssim{} &
      (\tau^{1/2} + h^2) \nm{g}_{
        L^2(\Omega;L^2(0,T;L^2(\mathcal O)))
      }.
    \end{aligned}
  \end{equation}
\end{lemma}

\subsection{Discretization of the adjoint equation}
\label{ssec:S1-S2}
For any $ g \in L_\mathbb F^2(\Omega;L^2(0,T;L^2(\mathcal O))) $, define
\[
  \mathcal S_1g \in \mathcal X_{h,\tau}
  \quad\text{and}\quad
  \mathcal S_2g \in L_\mathbb F^2(\Omega;L^2(0,T;\mathcal V_h))
\]
by
\begin{equation} 
  \label{eq:calS1-S2}
  \begin{aligned}
    & (\mathcal S_1g)_j - (\mathcal S_1g)_{j+1} \\
    ={} &
    \tau \Delta_h (\mathcal S_1g)_j +
    \int_{t_j}^{t_{j+1}} (\mathcal S_2g + Q_hg)(t) \, \mathrm{d}t -
    \int_{t_j}^{t_{j+1}} (\mathcal S_2g)(t) \, \mathrm{d}W(t)
  \end{aligned}
\end{equation}
for each $ 0 \leqslant j < J $, where $ (\mathcal S_1g)_J = 0 $.

We first establish the stability of $ \mathcal S_1 $.
\begin{lemma}
  \label{lem:S1-S2-stab}
  If $ g \in L_\mathbb F^2(\Omega;L^2(0,T;L^2(\mathcal O))) $, then
  \begin{equation} 
    \label{eq:S1-S2-stab}
    \begin{aligned}
      & \max_{0 \leqslant j \leqslant J}
      \nm{(\mathcal S_1g)_j}_{
        L^2(\Omega;L^2(\mathcal O))
      } + \nm{\mathcal S_1g}_{
        L^2(\Omega;L^2(0,T;\dot H^1(\mathcal O)))
      } \\
      \lesssim{} &
      \nm{g}_{
        L^2(\Omega;L^2(0,T;\dot H^{-1}(\mathcal O)))
      } + \tau \nm{g}_{
        L^2(\Omega;L^2(0,T;L^2(\mathcal O)))
      }.
    \end{aligned}
  \end{equation}
\end{lemma}
\begin{proof} 
  Define $ P := \mathcal S_1g $ and $ Z := \mathcal S_2g $. We divide the proof
  into the following four steps.

  {\it Step 1.} Let us prove, for any $ 0 \leqslant j < J $,
  \begin{equation} 
    \label{eq:Pj-I1-I2-I3}
    \nm{P_j}_{L^2(\Omega;L^2(\mathcal O))}^2 +
    \tau \nm{P_j}_{L^2(\mathcal O;\dot H^1(\mathcal O))}^2 =
    \mathbb I_1 + \mathbb I_2 + \mathbb I_3,
  \end{equation}
  where
  \begin{align*}
    \mathbb I_1 &:= [P_{j+1}, P_{j}(1+\delta W_j)], \\
    \mathbb I_2 &:= \Big[
      \int_{t_j}^{t_{j+1}} g(t) \, \mathrm{d}t, \, P_j(1+\delta W_j)
    \Big], \\
    \mathbb I_3 &:= \Big[
      \int_{t_j}^{t_{j+1}} Z(t) \, \mathrm{d}t, \, P_j \delta W_j
    \Big].
  \end{align*}
  By \cref{eq:calS1-S2} we have
  \begin{small}
  \begin{align*} 
    & [P_j - P_{j+1}, P_j] \\
    ={} &
    \tau [\Delta_h P_j, P_j] +
    \Big[
      \int_{t_j}^{t_{j+1}} (Z + g)(t) \, \mathrm{d}t,
      P_j
    \Big] - \Big[
      \int_{t_j}^{t_{j+1}} Z(t) \, \mathrm{d}W(t), P_j
    \Big] \\
    ={} &
    \tau [\Delta_h P_j, P_j] +
    \Big[
      \int_{t_j}^{t_{j+1}} (Z + g)(t) \, \mathrm{d}t,
      P_j
    \Big] - \Big[
      \mathbb E_{t_j} \Big(
        \int_{t_j}^{t_{j+1}} Z(t) \, \mathrm{d}W(t)
      \Big), P_j
    \Big] \\
    ={} &
    \tau [\Delta_h P_j, P_j] +
    \Big[
      \int_{t_j}^{t_{j+1}} (Z + g)(t) \, \mathrm{d}t,
      P_j
    \Big] \\
    ={} &
    -\tau \nm{P_j}_{L^2(\Omega;\dot H^1(\mathcal O))}^2 +
    \Big[ \int_{t_j}^{t_{j+1}} (Z+g)(t) \, \mathrm{d}t, P_j \Big],
  \end{align*}
  \end{small}
  so that
  \begin{align} 
    & \nm{P_j}_{L^2(\Omega;L^2(\mathcal O))}^2 +
    \tau \nm{P_j}_{L^2(\Omega;\dot H^1(\mathcal O))}^2 \notag \\
    ={} &
    [P_j, P_{j+1}] + \Big[
      \int_{t_j}^{t_{j+1}} (Z+g)(t) \, \mathrm{d}t, P_j
    \Big] \notag \\
    ={} &
    [P_j, P_{j+1}] + \Big[ \int_{t_j}^{t_{j+1}} Z(t) \, \mathrm{d}t, P_j\Big] +
    \Big[ \int_{t_j}^{t_{j+1}} g(t) \, \mathrm{d}t, P_j\Big].
    \label{eq:731}
  \end{align}
  Since \cref{eq:calS1-S2} implies
  \begin{equation}
    \label{eq:zq}
    \int_{t_j}^{t_{j+1}} Z(t) \, \mathrm{d}W(t) =
    (I-\mathbb E_{t_j}) \Big(
      P_{j+1} + \int_{t_j}^{t_{j+1}} (Z+Q_hg)(t) \, \mathrm{d}t
    \Big),
  \end{equation}
  we obtain
  \begin{align*} 
    & \Big[
      \int_{t_j}^{t_{j+1}} Z(t) \, \mathrm{d}t, P_j
    \Big] =
    \int_{t_j}^{t_{j+1}} [Z(t), P_j] \, \mathrm{d}t \\
    ={} &
    \Big[
      \int_{t_j}^{t_{j+1}} Z(t) \, \mathrm{d}W(t), \,
      \int_{t_j}^{t_{j+1}} P_j \, \mathrm{d}W(t)
    \Big] \\
    ={} &
    \Big[
      \int_{t_j}^{t_{j+1}} Z(t) \, \mathrm{d}W(t),
      P_j \delta W_j
    \Big] \\
    ={} &
    \Big[
      P_{j+1} + \int_{t_j}^{t_{j+1}} (Z+g)(t) \, \mathrm{d}t, \,
      P_j \delta W_j
    \Big].
  \end{align*}
  Inserting the above equality into \cref{eq:731} yields \cref{eq:Pj-I1-I2-I3}.

  {\it Step 2.} Let us estimate $ \mathbb I_1 $, $ \mathbb I_2 $ and $ \mathbb
  I_3 $. For $ \mathbb I_1 $ and $ \mathbb I_2 $, we have the following two
  estimates:
  \begin{align*} 
    \mathbb I_1 & \leqslant
    \nm{P_{j+1}}_{L^2(\Omega;L^2(\mathcal O))}
    \nm{P_j(1+\delta W_j)}_{L^2(\Omega;L^2(\mathcal O))} \\
    &= \sqrt{1+\tau} \nm{P_{j+1}}_{L^2(\Omega;L^2(\mathcal O))}
    \nm{P_j}_{L^2(\Omega;L^2(\mathcal O))} \\
    & \leqslant
    \frac12 \nm{P_{j+1}}_{L^2(\Omega;L^2(\mathcal O))}^2 +
    \frac{1+\tau}2 \nm{P_j}_{L^2(\Omega;L^2(\mathcal O))}^2
  \end{align*}
  and
  \begin{align*}
    \mathbb I_2 & \leqslant
    \sqrt\tau \nm{g}_{
      L^2(\Omega;L^2(t_j,t_{j+1};\dot H^{-1}(\mathcal O)))
    } \nm{P_j(1+\delta W_j)}_{
      L^2(\Omega;\dot H^1(\mathcal O))
    } \\
    &= \sqrt{\tau(1+\tau)} \nm{g}_{
      L^2(\Omega;L^2(t_j,t_{j+1};\dot H^{-1}(\mathcal O)))
    } \nm{P_j}_{L^2(\Omega;\dot H^1(\mathcal O))} \\
    & \leqslant
    \frac{1+\tau}2 \nm{g}_{
      L^2(\Omega;L^2(t_j,t_{j+1};\dot H^{-1}(\mathcal O)))
    }^2 + \frac\tau2
    \nm{P_j}_{L^2(\Omega;\dot H^1(\mathcal O))}^2.
  \end{align*}
  By \cref{eq:zq} we have
  \begin{align*}
    & \nm{Z}_{L^2(\Omega;L^2(t_j,t_{j+1};L^2(\mathcal O)))} =
    \Nm{
      \int_{t_j}^{t_{j+1}} Z(t) \, \mathrm{d}W(t)
    }_{L^2(\Omega;L^2(\mathcal O))} \\
    \leqslant{} &
    \nm{P_{j+1}}_{L^2(\Omega;L^2(\mathcal O))} +
    \Nm{
      \int_{t_j}^{t_{j+1}} (Z+Q_hg)(t) \, \mathrm{d}t
    }_{L^2(\Omega;L^2(\mathcal O))} \\
    \leqslant{} &
    \nm{P_{j+1}}_{L^2(\Omega;L^2(\mathcal O))} +
    \sqrt\tau \nm{Z+Q_hg}_{L^2(\Omega;L^2(t_j,t_{j+1};L^2(\mathcal O)))} \\
    \leqslant{} &
    \nm{P_{j+1}}_{L^2(\Omega;L^2(\mathcal O))} +
    \sqrt\tau \nm{Z}_{L^2(\Omega;L^2(t_j,t_{j+1};L^2(\mathcal O)))} \\
    & \quad {} +
    \sqrt\tau \nm{g}_{L^2(\Omega;L^2(t_j,t_{j+1};L^2(\mathcal O)))},
  \end{align*}
  so that
  \begin{equation*} 
    \begin{aligned}
      & \nm{Z}_{L^2(\Omega;L^2(t_j,t_{j+1};L^2(\mathcal O)))} \\
      \leqslant{} &
      \frac1{1-\sqrt\tau} \big(
        \nm{P_{j+1}}_{L^2(\Omega;L^2(\mathcal O))} +
        \sqrt\tau\nm{g}_{L^2(\Omega;L^2(t_j,t_{j+1};L^2(\mathcal O)))}
      \big).
    \end{aligned}
  \end{equation*}
  It follows that
  \begin{equation} 
    \label{eq:2000}
    \begin{aligned}
      & \nm{Z}_{L^2(\Omega;L^2(t_j,t_{j+1};L^2(\mathcal O)))}^2 \\
      \leqslant{} &
      \frac2{(1-\sqrt\tau)^2} \big(
        \nm{P_{j+1}}_{L^2(\Omega;L^2(\mathcal O))}^2 +
        \tau\nm{g}_{L^2(\Omega;L^2(t_j,t_{j+1};L^2(\mathcal O)))}^2
      \big).
    \end{aligned}
  \end{equation}
  By the definition of $ \mathbb I_3 $, we obtain
  \begin{align}
    \mathbb I_3 & \leqslant
    \sqrt\tau \nm{Z}_{L^2(\Omega;L^2(t_j,t_{j+1};L^2(\mathcal O)))}
    \nm{P_j \delta W_j}_{L^2(\Omega;L^2(\mathcal O))} \notag \\
    &= \tau \nm{Z}_{L^2(\Omega;L^2(t_j,t_{j+1};L^2(\mathcal O)))}
    \nm{P_j}_{L^2(\Omega;L^2(\mathcal O))} \notag \\
    & \leqslant \frac\tau2 \nm{Z}_{L^2(\Omega;L^2(t_j,t_{j+1};L^2(\mathcal O)))}^2 +
    \frac\tau2 \nm{P_j}_{L^2(\Omega;L^2(\mathcal O))}^2.
    \label{eq:732}
  \end{align}
  Combining \cref{eq:2000,eq:732} yields
  \begin{align*}
    \mathbb I_3 & \leqslant
    \frac\tau{(1-\sqrt\tau)^2} \nm{P_{j+1}}_{L^2(\Omega;L^2(\mathcal O))}^2 +
    \frac{\tau^2}{(1-\sqrt\tau)^2}
    \nm{g}_{L^2(\Omega;L^2(t_j,t_{j+1}; L^2(\mathcal O)))}^2 \\
    & \qquad {} +
    \frac\tau2 \nm{P_j}_{L^2(\Omega;L^2(\mathcal O))}^2.
  \end{align*}

  {\it Step 3.} Combining \cref{eq:Pj-I1-I2-I3} and the estimates of $ \mathbb
  I_1 $, $ \mathbb I_2 $ and $ \mathbb I_3 $ in Step 2 yields
  \begin{align*} 
    & \nm{P_j}_{L^2(\Omega;L^2(\mathcal O))}^2 +
    \tau \nm{P_j}_{L^2(\Omega;\dot H^1(\mathcal O))}^2 \notag \\
    \leqslant{} & \Big(
      \frac12 + \frac\tau{(1-\sqrt\tau)^2}
    \Big) \nm{P_{j+1}}_{L^2(\Omega;L^2(\mathcal O))}^2 +
    \frac{1+2\tau}2 \nm{P_j}_{L^2(\Omega;L^2(\mathcal O))}^2 \notag \\
    & {} + \frac\tau2 \nm{P_j}_{L^2(\Omega;\dot H^1(\mathcal O))}^2 +
    \frac{1+\tau}2 \nm{g}_{
      L^2(\Omega;L^2(t_j,t_{j+1};\dot H^{-1}(\Omega)))
    }^2 \notag \\
    & {} + \frac{\tau^2}{(1-\sqrt\tau)^2}
    \nm{g}_{L^2(\Omega;L^2(t_j,t_{j+1};L^2(\mathcal O)))}^2,
  \end{align*}
  so that
  \begin{small}
  \begin{align} 
    & \nm{P_j}_{L^2(\Omega;L^2(\mathcal O))}^2 +
    \frac{\tau}{1-2\tau} \nm{P_j}_{L^2(\Omega;\dot H^1(\mathcal O))}^2 \notag \\
    \leqslant{} &
    \frac1{1-2\tau}\Big(
      1 + \frac{2\tau}{(1-\sqrt\tau)^2}
    \Big) \nm{P_{j+1}}_{L^2(\Omega;L^2(\mathcal O))}^2 +
    \frac{1+\tau}{1-2\tau} \nm{g}_{
      L^2(\Omega;L^2(t_j,t_{j+1};\dot H^{-1} (\mathcal O)))
    }^2 \notag \\
    & {} + \frac{2\tau^2}{(1-2\tau)(1-\sqrt\tau)^2}
    \nm{g}_{L^2(\Omega;L^2(t_j,t_{j+1};L^2(\mathcal O)))}^2.
    \label{eq:tmp}
  \end{align}
  \end{small}
  Hence, applying the discrete Gronwall's inequality yields, by the fact $ P_J =
  0 $, that
  \begin{small}
  \begin{align*} 
    & \max_{0 \leqslant j \leqslant J}
    \nm{P_j}_{L^2(\Omega;L^2(\mathcal O))}^2 \\
    \leqslant{} & \Big( \frac1{1-2\tau} \Big)^J \Big(
      1 + \frac{2\tau}{(1-\sqrt\tau)^2}
    \Big)^{J-1} \times {} \\
    & \Big(
      (1+\tau) \nm{g}_{
        L^2(\Omega;L^2(0,T;\dot H^{-1}(\mathcal O)))
      }^2 + \frac{2\tau^2}{(1-\sqrt\tau)^2}
      \nm{g}_{L^2(\Omega;L^2(0,T;L^2(\mathcal O)))}^2
    \Big) \\
    \lesssim{} &
    \nm{g}_{L^2(\Omega;L^2(0,T;\dot H^{-1}(\mathcal O)))}^2 +
    \tau^2 \nm{g}_{L^2(\Omega;L^2(0,T;L^2(\mathcal O)))}^2,
  \end{align*}
  \end{small}
  which implies
  \begin{equation}
    \label{eq:P-max}
    \max_{0 \leqslant j \leqslant J}
    \nm{P_j}_{L^2(\Omega;L^2(\mathcal O))}
    \lesssim \nm{g}_{L^2(\Omega;L^2(0,T;\dot H^{-1}(\mathcal O)))} +
    \tau \nm{g}_{L^2(\Omega;L^2(0,T;L^2(\mathcal O)))}.
  \end{equation}

  {\it Step 4.} Summing both sides of \cref{eq:tmp} over $ j $ from $ 0 $ to $
  J-1 $ yields
  \begin{small}
  \begin{align*} 
    & \sum_{j=0}^{J-1} \nm{P_j}_{L^2(\Omega;L^2(\mathcal O))}^2 +
    \frac1{1-2\tau}
    \nm{P}_{L^2(\Omega;L^2(0,T;\dot H^1(\mathcal O)))}^2 \\
    \leqslant{} &
    \frac1{1-2\tau}\Big(
      1+\frac{2\tau}{(1-\sqrt\tau)^2}
    \Big)  \sum_{j=0}^{J-1} \nm{P_{j+1}}_{L^2(\Omega;L^2(\mathcal O))}^2 +
    \frac{1+\tau}{1-2\tau} \nm{g}_{
      L^2(\Omega;L^2(0,T;\dot H^{-1}(\mathcal O)))
    }^2 \\
    & \qquad {} + \frac{2\tau^2}{(1-2\tau)(1-\sqrt\tau)^2}
    \nm{g}_{L^2(\Omega;L^2(0,T;L^2(\mathcal O)))}^2.
  \end{align*}
  \end{small}
  It follows, by the fact $ P_J = 0 $, that
  \begin{small}
  \begin{align*}
    & \frac1{1-2\tau} \nm{P}_{
      L^2(\Omega;L^2(0,T;\dot H^1(\mathcal O)))
    }^2 \\
    \leqslant{} &
    \Big(
      \frac1{1-2\tau}\Big(
        1 + \frac{2\tau}{(1-\sqrt\tau)^2}
      \Big) - 1
    \Big) \sum_{j=1}^{J-1} \nm{P_j}_{L^2(\Omega;L^2(\mathcal O))}^2 +
    \frac{1+\tau}{1-2\tau} \nm{g}_{
      L^2(\Omega;L^2(0,T;\dot H^{-1}(\mathcal O)))
    }^2 \\
    & {} + \frac{2\tau^2}{(1-2\tau)(1-\sqrt\tau)^2}
    \nm{g}_{L^2(\Omega;L^2(0,T;L^2(\mathcal O)))}^2 \\
    \lesssim{} &
    \max_{1 \leqslant j < J} \nm{P_j}_{L^2(\Omega;L^2(\mathcal O))}^2 +
    \nm{g}_{L^2(\Omega;L^2(0,T;\dot H^{-1}(\mathcal O)))}^2 +
    \tau^2 \nm{g}_{L^2(\Omega;L^2(0,T;L^2(\mathcal O)))}^2.
  \end{align*}
  \end{small}
  Hence, we conclude from \cref{eq:P-max} that
  \[ 
    \nm{P}_{L^2(\Omega;L^2(0,T;\dot H^1(\mathcal O)))}
    \lesssim \nm{g}_{L^2(\Omega;L^2(0,T;\dot H^{-1}(\mathcal O)))} +
    \tau \nm{g}_{L^2(\Omega;L^2(0,T;L^2(\mathcal O)))}.
  \]
  Finally, combining the above estimate and \cref{eq:P-max} proves
  \cref{eq:S1-S2-stab}.
\end{proof}
\begin{remark}
  Note that \cref{eq:2000} implies, for any $ 0 \leqslant j < J $.
  \[
    \nm{\mathcal S_2g}_{
      L^2(\Omega;L^2(t_j,t_{j+1};L^2(\mathcal O)))
    }^2 \lesssim \nm{(\mathcal S_1g)_{j+1}}_{
      L^2(\Omega;L^2(\mathcal O))
    }^2 + \tau \nm{g}_{
      L^2(\Omega;L^2(t_j,t_{j+1};L^2(\mathcal O)))
    }^2.
  \]
  It follows that
  \begin{small}
  \[
    \nm{\mathcal S_2g}_{
      L^2(\Omega;L^2(0,T;L^2(\mathcal O)))
    }^2 \lesssim \tau^{-1} \max_{0 \leqslant j \leqslant J}
    \nm{(\mathcal S_1g)_j}_{L^2(\Omega;L^2(\mathcal O))}^2 +
    \tau \nm{g}_{
      L^2(\Omega;L^2(0,T;L^2(\mathcal O)))
    }^2.
  \]
\end{small}
  Hence, by \cref{eq:S1-S2-stab} we obtain
  \begin{equation}
    \label{eq:calS2g-stab}
    \begin{aligned}
      & \nm{\mathcal S_2g}_{
        L^2(\Omega;L^2(0,T;L^2(\mathcal O)))
      } \\
      \lesssim{} &
      \tau^{-1/2} \nm{g}_{
        L^2(\Omega;L^2(0,T;\dot H^{-1}(\mathcal O)))
      } + \nm{g}_{
        L^2(\Omega;L^2(0,T;L^2(\mathcal O)))
      }.
    \end{aligned}
  \end{equation}
\end{remark}

Then we obtain the stability of $ \mathcal S_2 $.
\begin{lemma} 
  \label{lem:S2-stab}
  If $ g \in L_\mathbb F^2(\Omega;L^2(0,T;L^2(\mathcal O))) $, then
  \begin{equation} 
    \label{eq:S2-stab}
    \nm{\mathcal S_2g}_{L^2(\Omega;L^2(0,T;L^2(\mathcal O)))}
    \lesssim \nm{g}_{
      L^2(\Omega;L^2(0,T;L^2(\mathcal O)))
    }.
  \end{equation}
\end{lemma}
\begin{proof} 
  Let $ (p_h, z_h) $ be the solution of the equation
  \begin{equation} 
    \begin{cases}
      \mathrm{d}p_h(t) =
      -(\Delta_h p_h + Q_hg + z_h)(t) \, \mathrm{d}t
      + z_h(t) \, \mathrm{d}W(t), \,\, 0 \leqslant t \leqslant T, \\
      p_h(T) = 0.
    \end{cases}
  \end{equation}
  Similarly to \cref{eq:S1-S2-regu}, we have
  \begin{equation} 
    \label{eq:zh-regu}
    \nm{z_h}_{L^2(\Omega;L^2(0,T;\dot H_h^1(\mathcal O)))}
    \lesssim \nm{g}_{
      L^2(\Omega;L^2(0,T;L^2(\mathcal O)))
    }.
  \end{equation}
  Define $ P \in \mathcal X_{h,\tau} $ and $ Z \in L_\mathbb F^2(\Omega;L^2(0,T;
  \mathcal V_h)) $ by
  \begin{small}
  \begin{equation} 
    \label{eq:S2-stab-1}
    \begin{cases}
      P_j \!-\! P_{j+1} =
      \tau\Delta_h P_j \!+\!
      Q_h \int_{t_j}^{t_{j+1}} (g \!+\! z_h)(t) \, \mathrm{d}t -
      \int_{t_j}^{t_{j+1}} Z(t) \, \mathrm{d}W(t),
      \,\, 0 \leqslant j < J, \\
      P_J = 0.
    \end{cases}
  \end{equation}
  \end{small}
  By \cref{eq:zh-Z} we obtain
  \begin{equation}
    \label{eq:zh-wtZ}
    \nm{z_h-Z}_{
      L^2(\Omega;L^2(0,T;\dot H_h^{-1}(\mathcal O)))
    }  \lesssim \tau^{1/2} \nm{g+z_h}_{
      L^2(\Omega;L^2(0,T;L^2(\mathcal O)))
    },
  \end{equation}
  and by \cref{rem:Z-stab} we have
  \begin{equation}
    \label{eq:97}
    \nm{Z}_{
      L^2(\Omega;L^2(0,T;L^2(\mathcal O)))
    } \lesssim \nm{g+z_h}_{
      L^2(\Omega;L^2(0,T;L^2(\mathcal O)))
    }.
  \end{equation}
  Let
  \[ 
    E^P := \mathcal S_1g - P \quad \text{and} \quad
    E^Z := \mathcal S_2g - Z.
  \]
  By \cref{eq:calS1-S2,eq:S2-stab-1} we have, for any $ 0 \leqslant j < J $,
  \begin{align*} 
    & E_j^P - E_{j+1}^P \\
    ={} &
    \tau \Delta_h E_j^P  +
    \int_{t_j}^{t_{j+1}} (\mathcal S_2g-z_h)(t) \, \mathrm{d}t -
    \int_{t_j}^{t_{j+1}} E^Z(t) \, \mathrm{d}W(t) \\
    ={} &
    \tau \Delta_h E_j^P  +
    \int_{t_j}^{t_{j+1}} (E^Z + Z - z_h)(t) \, \mathrm{d}t -
    \int_{t_j}^{t_{j+1}} E^Z(t) \, \mathrm{d}W(t),
  \end{align*}
  and so from the fact $ E^P_J = 0 $ we conclude that
  \begin{equation}
    \label{eq:E^P-S1}
    E^P = \mathcal S_1(Z - z_h)
    \quad\text{and}\quad
    E^Z = \mathcal S_2(Z - z_h).
  \end{equation}
  It follows that
  \begin{align*} 
    & \nm{E^Z}_{
      L^2(\Omega;L^2(0,T;L^2(\mathcal O)))
    } \\
    \lesssim{} &
    \tau^{-1/2} \nm{z_h-Z}_{
      L^2(\Omega;L^2(0,T;\dot H^{-1}(\mathcal O)))
    } + \nm{z_h-Z}_{
      L^2(\Omega;L^2(0,T;L^2(\mathcal O)))
    } \quad\text{(by \cref{eq:calS2g-stab})} \\ \label{eq:1001}
    \lesssim{} &
    \tau^{-1/2} \nm{z_h-Z}_{
      L^2(\Omega;L^2(0,T;\dot H_h^{-1}(\mathcal O)))
    } + \nm{z_h-Z}_{
      L^2(\Omega;L^2(0,T;L^2(\mathcal O)))
    } \quad\text{(by \cref{eq:H-Hh-equiv}).}
  \end{align*}
  Therefore,
  \begin{align*}
    & \nm{\mathcal S_2g}_{
      L^2(\Omega;L^2(0,T;L^2(\mathcal O)))
    } \\
    \leqslant{} &
    \nm{E^Z}_{L^2(\Omega;L^2(0,T;L^2(\mathcal O)))} +
    \nm{Z}_{L^2(\Omega;L^2(0,T;L^2(\mathcal O)))} \\
    \lesssim{} &
    \tau^{-1/2} \nm{z_h-Z}_{
      L^2(\Omega;L^2(0,T;\dot H_h^{-1}(\mathcal O)))
    } +
    \nm{z_h-Z}_{L^2(\Omega;L^2(0,T;L^2(\mathcal O)))} \\
    & \qquad {} + \nm{Z}_{L^2(\Omega;L^2(0,T;L^2(\mathcal O)))} \\
    \lesssim{} &
    \nm{g}_{L^2(\Omega;L^2(0,T;L^2(\mathcal O)))}
    \quad\text{(by \cref{eq:zh-regu,eq:zh-wtZ,eq:97}).}
  \end{align*}
  This proves \cref{eq:S2-stab} and thus concludes the proof.
\end{proof}

Finally, by the stability estimate \cref{eq:S1-S2-stab} and the theoretical
results in \cref{ssec:sbde}, we derive the convergence of $ \mathcal S_1 $.
\begin{lemma}
  \label{lem:S1-conv}
  If $ g \in L_{\mathbb F}^2( \Omega;L^2(0,T;L^2(\mathcal O))) $, then
  \begin{equation} 
    \label{eq:S1-conv}
    \nm{S_1 g - \mathcal S_1g}_{
      L^2(\Omega;L^2(0,T;L^2(\mathcal O)))
    } \lesssim (\tau^{1/2} + h^2) \nm{g}_{
      L^2(\Omega;L^2(0,T;L^2(\mathcal O)))
    }.
  \end{equation}
\end{lemma}
\begin{proof} 
  Define $ P \in \mathcal X_{h,\tau} $ and $ Z \in L_\mathbb
  F^2(\Omega;L^2(0,T;\mathcal V_h)) $ by
  \begin{small}
  \begin{equation} 
    \label{eq:Conv-S1-1}
    \begin{cases}
      P_j - P_{j+1} =
      \tau\Delta_h P_j +
      Q_h \! \int_{t_j}^{t_{j+1}} \!
      (g \!+\! S_2g)(t) \, \mathrm{d}t \!-\!
      \int_{t_j}^{t_{j+1}} \! Z(t)
      \, \mathrm{d}W\!(t),
      \, 0 \leqslant j < J, \\
      P_J = 0.
    \end{cases}
  \end{equation}
  \end{small}
  By \cref{eq:first-p-P} we obtain
  \begin{align}
    & \nm{S_1g - P}_{
      L^2(\Omega;L^2(0,T;L^2(\mathcal O)))
    } + \nm{S_2g - Z}_{
      L^2(\Omega;L^2(0,T;\dot H^{-1}(\mathcal O)))
    } \notag \\
    \lesssim{} &
    (\tau^{1/2} + h^2) \nm{g+S_2g}_{
      L^2(\Omega;L^2(0,T;L^2(\mathcal O)))
    }, \label{eq:S2g-wtZ}
  \end{align}
  and by \cref{rem:Z-stab} we have
  \begin{align}
    \nm{Z}_{
      L^2(\Omega;L^2(0,T;L^2(\mathcal O)))
    } & \lesssim \nm{g+S_2g}_{
      L^2(\Omega;L^2(0,T;L^2(\mathcal O)))
    }. \label{eq:wtZ-stab}
  \end{align}
  Similar to \cref{eq:E^P-S1}, we have
  \[
    \mathcal S_1g - P = \mathcal S_1(Z - S_2g),
  \]
  and so \cref{eq:S1-S2-stab} implies
  \begin{align*} 
    & \max_{0 \leqslant j \leqslant J}
    \nm{E_j^P}_{L^2(\Omega;L^2(\mathcal O))} \\
    \lesssim{} &
    \nm{S_2g-Z}_{
      L^2(\Omega;L^2(0,T;\dot H^{-1}(\mathcal O)))
    } + \tau \nm{S_2g-Z}_{
      L^2(\Omega;L^2(0,T;L^2(\mathcal O)))
    }.
  \end{align*}
  It follows that
  \begin{align*} 
    & \nm{E^P}_{L^2(\Omega;L^2(0,T;L^2(\mathcal O)))} \\
    \lesssim{} &
    \nm{S_2g-Z}_{
      L^2(\Omega;L^2(0,T;\dot H^{-1}(\mathcal O)))
    } + \tau \nm{S_2g-Z}_{
      L^2(\Omega;L^2(0,T;L^2(\mathcal O)))
    },
  \end{align*}
  and hence
  \begin{align*} 
    & \nm{S_1g-\mathcal S_1g}_{
      L^2(\Omega;L^2(0,T;L^2(\mathcal O)))
    } \\
    \leqslant{} &
    \nm{S_1g - P}_{
      L^2(\Omega;L^2(0,T;L^2(\mathcal O)))
    } + \nm{E^P}_{
      L^2(\Omega;L^2(0,T;L^2(\mathcal O)))
    } \\
    \lesssim{} &
    \nm{S_1g-P}_{
      L^2(\Omega;L^2(0,T;L^2(\mathcal O)))
    } + \nm{S_2g-Z}_{
      L^2(\Omega;L^2(0,T;\dot H^{-1}(\mathcal O)))
    } \\
    & \quad {} + \tau \nm{S_2g-Z}_{
      L^2(\Omega;L^2(0,T;L^2(\mathcal O)))
    }.
  \end{align*}
  Therefore, \cref{eq:S1-conv} follows from \cref{eq:S2g-wtZ},
  \cref{eq:wtZ-stab} and the fact that
  \[
    \nm{S_2g}_{L^2(\Omega;L^2(0,T;\dot H^1(\mathcal O)))}
    \lesssim \nm{g}_{L^2(\Omega;L^2(0,T;L^2(\mathcal O)))}
    \quad\text{(by \cref{eq:S1-S2-regu}).}
  \]
  This completes the proof.
\end{proof}

\begin{remark}
  Assume that $ g \in L_{\mathbb F}^2(\Omega;L^2(0,T;L^2(\mathcal O))) $. By
  \cref{lem:dg-stab} we obtain
  \begin{align*} 
    & \tau^{-1} \sum_{j=0}^{J-1}
    \nm{
      (\mathcal S_1g)_j - (\mathcal S_1g)_{j+1}
    }_{L^2(\Omega;L^2(\mathcal O))}^2 \\
    \lesssim{} &
    \nm{g+\mathcal S_2g}_{L^2(\Omega;L^2(0,T;L^2(\mathcal O)))}^2 +
    \sum_{j=0}^{J-1} \tau^{-1}
    \Nm{
      \int_{t_j}^{t_{j+1}} (\mathcal S_2g)(t) \, \mathrm{d}W(t)
    }_{L^2(\Omega;L^2(\mathcal O))}^2 \\
    ={} &
    \nm{g+\mathcal S_2g}_{L^2(\Omega;L^2(0,T;L^2(\mathcal O)))}^2 +
    \sum_{j=0}^{J-1} \tau^{-1} \nm{\mathcal S_2g}_{
      L^2(\Omega;L^2(t_j,t_{j+1};L^2(\mathcal O)))
    }^2 \\
    ={} &
    \nm{g+\mathcal S_2g}_{
      L^2(\Omega;L^2(0,T;L^2(\mathcal O)))
    }^2 + \tau^{-1} \nm{\mathcal S_2g}_{
      L^2(\Omega;L^2(0,T;L^2(\mathcal O)))
    }^2,
  \end{align*}
  so that
  \begin{align*} 
    & \sum_{j=0}^{J-1} \nm{
      (\mathcal S_1g)_j - (\mathcal S_1g)_{j+1}
    }_{L^2(\Omega;L^2(\mathcal O))}^2 \\
    \lesssim{} &
    \tau\nm{g}_{L^2(\Omega;L^2(0,T;L^2(\mathcal O)))}^2 +
    \nm{\mathcal S_2g}_{L^2(\Omega;L^2(0,T;L^2(\mathcal O)))}^2 \\
    \lesssim{} &
    \nm{g}_{L^2(\Omega;L^2(0,T;L^2(\mathcal O)))}^2
    \quad\text{(by \cref{lem:S2-stab}).}
  \end{align*}
  By this estimate and \cref{eq:S1-conv} we then obtain
  \begin{align*} 
    & \sum_{j=0}^{J-1} \nm{S_1g-(\mathcal S_1g)_{j+1}}_{
      L^2(\Omega;L^2(t_j,t_{j+1};L^2(\mathcal O)))
    }^2 \\
    \lesssim{} &
    \sum_{j=0}^{J-1} \big(
      \nm{S_1g-(\mathcal S_1g)_{j}}_{
        L^2(\Omega;L^2(t_j,t_{j+1};L^2(\mathcal O)))
      }^2 + \tau \nm{
        (\mathcal S_1g)_j - (\mathcal S_1g)_{j+1}
      }_{L^2(\Omega;L^2(\mathcal O))}^2
    \big) \\
    \lesssim{} &
    (\tau^{1/2} + h^2)^2 \nm{g}_{
      L^2(\Omega;L^2(0,T;L^2(\mathcal O)))
    }^2,
  \end{align*}
  which implies
  \begin{small}
  \begin{equation} 
    \label{eq:S1-conv-new}
    \Big(
      \sum_{j=0}^{J-1} \nm{S_1g-(\mathcal S_1g)_{j+1}}_{
        L^2(\Omega;L^2(t_j,t_{j+1};L^2(\mathcal O)))
      }
    \Big)^{1/2} \lesssim
    (\tau^{1/2} + h^2) \nm{g}_{
      L^2(\Omega;L^2(0,T;L^2(\mathcal O)))
    }.
  \end{equation}
\end{small}
\end{remark}

\subsection{Proof of \texorpdfstring{\cref{thm:conv}}{}}
\label{ssec:proof-conv}
In this subsection, we will use the theoretical results in
\cref{ssec:S0,ssec:S1-S2} to prove \cref{thm:conv}. The basic idea is to use the
first order optimality conditions of problem \cref{eq:model} and
\cref{eq:discretization} to derive the error estimate; this is standard in the
numerical analysis of the optimal control problems with PDE constraints (see,
e.g., \cite{Hinze2005}).

We first present the following four auxiliary lemmas.
\begin{lemma}
  \label{lem:S0f-g}
  For any $ f, g \in L_{\mathbb F}^2(\Omega;L^2(0,T;L^2(\mathcal O))) $, we have
  \begin{equation} 
    \label{eq:S1f-g}
    \begin{aligned}
      \int_0^T [\mathcal S_0 f(t), g(t)] \, \mathrm{d}t
      & =
      \sum_{j=0}^{J-1} \bigg( \Big[
        \int_{t_j}^{t_{j+1}} f(t) \, \mathrm{d}t, \,
        (\mathcal S_1 g)_{j+1}
      \Big] - {} \\
      & \qquad\qquad\qquad \Big[
        (\mathcal S_0 f)_j \delta W_j, \,
        \int_{t_j}^{t_{j+1}} (\mathcal S_2g + g)(t) \, \mathrm{d}t
      \Big]
    \bigg).
    \end{aligned}
  \end{equation}
\end{lemma}
\begin{proof} 
  By \cref{eq:calS0,eq:calS1-S2}, a straightforward computation yields
  \begin{align*}
    & \sum_{j=0}^{J-1} \big[
      (\mathcal S_0 f)_{j+1} - (\mathcal S_0 f)_j,
      (\mathcal S_1 g)_{j+1}
    \big] \\
    ={} &
    \sum_{j=0}^{J-1} \Big[
      \tau\Delta_h (\mathcal S_0 f)_{j+1} +
      \int_{t_j}^{t_{j+1}} f(t) \, \mathrm{d}t +
      (\mathcal S_0 f)_j \delta W_j, \,
      (\mathcal S_1 g)_{j+1}
    \Big]
  \end{align*}
  and
  \begin{align*} 
    & \sum_{j=0}^{J-1} \big[
      (\mathcal S_1 g)_j-(\mathcal S_1 g)_{j+1},
      \, (\mathcal S_0 f)_{j}
    \big] \\
    ={} &
    \sum_{j=0}^{J-1} \Big[
      \tau \Delta_h (\mathcal S_1 g)_j +
      \int_{t_j}^{t_{j+1}} (\mathcal S_2g + g)(t) \, \mathrm{d}t -
      \int_{t_j}^{t_{j+1}} \mathcal S_2g(t) \, \mathrm{d}W(t),
      \, (\mathcal S_0 f)_{j}
    \Big] \\
    ={} &
    \sum_{j=0}^{J-1} \Big[
      \tau\Delta_h (\mathcal S_1 g)_j +
      \int_{t_j}^{t_{j+1}} (\mathcal S_2g + g)(t) \, \mathrm{d}t,
      \, (\mathcal S_0 f)_{j}
    \Big].
  \end{align*}
  Since a simple calculation gives, by the facts $ (\mathcal S_0 f)_0 = 0 $ and
  $ (\mathcal S_1 g)_J = 0 $, that
  \begin{align*} 
    \sum_{j=0}^{J-1} \big[
      (\mathcal S_0 f)_{j+1} - (\mathcal S_0 f)_j,
      (\mathcal S_1 g)_{j+1}
    \big] =
    \sum_{j=0}^{J-1} \big[
      (\mathcal S_0 f)_{j},
      \, (\mathcal S_1 g)_j-(\mathcal S_1 g)_{j+1}
    \big],
  \end{align*}
  we then obtain
  \begin{align*} 
    & \sum_{j=0}^{J-1} \Big[
      \tau\Delta_h (\mathcal S_0f)_{j+1} +
      \int_{t_j}^{t_{j+1}}f(t) \, \mathrm{d}t +
      (\mathcal S_0f)_{j} \delta W_j,
      (\mathcal S_1g)_{j+1}
    \Big] \\
    ={} &
    \sum_{j=0}^{J-1} \Big[
      \tau\Delta_h (\mathcal S_1g)_j +
      \int_{t_j}^{t_{j+1}} (\mathcal S_2g + g)(t) \, \mathrm{d}t,
      \, (\mathcal S_0 f)_{j}
    \Big].
  \end{align*}
  As $ [\Delta_h (\mathcal S_0 f)_{i}, (\mathcal S_1 g)_{i}] = [\Delta_h
  (\mathcal S_1 g)_i, (\mathcal S_0 f)_{i}] $ for all $ 1 \leqslant i \leqslant
  J $ and $ (\mathcal S_0 f)_0 = (\mathcal S_1 g)_J = 0 $, it follows that
  \begin{align*} 
    & \sum_{j=0}^{J-1} \Big[
      \int_{t_j}^{t_{j+1}} f(t) \, \mathrm{d}t +
      (\mathcal S_0f)_{j} \delta W_j,
      (\mathcal S_1g)_{j+1}
    \Big] \\
    ={} &
    \sum_{j=0}^{J-1} \Big[
      \int_{t_j}^{t_{j+1}} (\mathcal S_2g + g)(t) \, \mathrm{d}t,
      \, (\mathcal S_0 f)_{j}
    \Big].
  \end{align*}
  Hence, a direct calculation gives
  \begin{equation} 
    \label{eq:200}
    \begin{aligned}
      \sum_{j=0}^{J-1} \Big[
        \int_{t_j}^{t_{j+1}} g(t) \, \mathrm{d}t,
        \, (\mathcal S_0 f)_{j}
      \Big] & =
      \sum_{j=0}^{J-1} \Big(
        \Big[
          \int_{t_j}^{t_{j+1}} f(t) \, \mathrm{d}t,
          \, (\mathcal S_1g)_{j+1}
        \Big] + {} \\
        & \big[
          (\mathcal S_1 g)_{j+1} \delta W_j -
          \int_{t_j}^{t_{j+1}} \mathcal S_2g(t) \, \mathrm{d}t,
          \, (\mathcal S_0 f)_{j}
        \big]
      \Big).
    \end{aligned}
  \end{equation}
  On the other hand, from \cref{eq:calS1-S2} we get
  \begin{align*}
    & -\big[
      (\mathcal S_1g)_{j+1},
      (\mathcal S_0f)_{j} \delta W_j
    \big] \\
    ={} &
    \Big[
      \int_{t_j}^{t_{j+1}} (\mathcal S_2g + g)(t) \, \mathrm{d}t -
      \int_{t_j}^{t_{j+1}} \mathcal S_2g(t) \, \mathrm{d}W(t),
      \, (\mathcal S_0f)_{j} \delta W_j
    \Big] \\
    ={} &
    \Big[
      \int_{t_j}^{t_{j+1}} (\mathcal S_2 g + g)(t) \, \mathrm{d}t,
      \, (\mathcal S_0f)_{j} \delta W_j
    \Big] -
    \Big[
      \int_{t_j}^{t_{j+1}} \mathcal S_2g(t) \, \mathrm{d}t,
      \, (\mathcal S_0f)_{j}
    \Big],
  \end{align*}
  so that
  \begin{align*}
    & \Big[
      (\mathcal S_1g)_{j+1} \delta W_j -
      \int_{t_j}^{t_{j+1}} \mathcal S_2g(t) \, \mathrm{d}t,
      \, (\mathcal S_0f)_{j}
    \Big] \\
    ={} &
    -\Big[
      \int_{t_j}^{t_{j+1}} (\mathcal S_2g + g)(t) \, \mathrm{d}t,
      \, (\mathcal S_0f)_{j} \delta W_j
    \Big].
  \end{align*}
  Inserting the above equality into \cref{eq:200} yields
  \begin{align*}
    & \sum_{j=0}^{J-1} \Big[
      (\mathcal S_0f)_{j}, \,
      \int_{t_j}^{t_{j+1}} g(t) \, \mathrm{d}t
    \Big] \\
    ={} &
    \sum_{j=0}^{J-1} \Big(
      \Big[
        \int_{t_j}^{t_{j+1}} f(t) \, \mathrm{d}t,
        \, (\mathcal S_1g)_{j+1}
      \Big] - \Big[
        (\mathcal S_0f)_{j} \delta W_j, \,
        \int_{t_j}^{t_{j+1}} (\mathcal S_2g + g)(t) \, \mathrm{d}t
      \Big]
    \Big).
  \end{align*}
  Therefore, \cref{eq:S1f-g} follows from the fact
  \[
    \int_0^T [\mathcal S_0 f(t), g(t)] \, \mathrm{d}t =
    \sum_{j=0}^{J-1} \Big[
      (\mathcal S_0f)_j, \,
      \int_{t_j}^{t_{j+1}} g(t) \, \mathrm{d}t
    \Big].
  \]
  This completes the proof.
\end{proof}

\begin{lemma}
  \label{lem:p-pj}
  For any $ g \in L_\mathbb F^2(\Omega;L^2(0,T;L^2(\mathcal O))) $,
  \begin{equation}
    \label{eq:p-pj}
    \sum_{j=0}^{J-1}
    \nm{S_1g - (S_1g)(t_j)}_{
      L^2(\Omega;L^2(t_j,t_{j+1};L^2(\mathcal O)))
    }^2 \lesssim \tau
    \nm{g}_{L^2(\Omega;L^2(0,T;L^2(\mathcal O)))}^2.
  \end{equation}
\end{lemma}
\noindent The proof of this lemma is omitted here, since it is similar to that
of \cref{lem:y-yj}; see also the proof of \cref{eq:ph-P-l2}.
\begin{lemma}
  \label{lem:ljm}
  For any $ f,g \in L_\mathbb F^2(\Omega;L^2(0,T;L^2(\mathcal O))) $, we have
  \begin{equation}
    \label{eq:ljm}
    \begin{aligned}
      & \sum_{j=0}^{J-1} \Big[
        \int_{t_j}^{t_{j+1}} (\mathcal S_2 f + f)(t) \, \mathrm{d}t, \,
        (\mathcal S_0 g)_j \delta W_j
      \Big] \\
      \lesssim{} &
      \tau^{1/2} \nm{f}_{
        L^2(\Omega;L^2(0,T;L^2(\mathcal O)))
      } \nm{g}_{
        L^2(\Omega;L^2(0,T;L^2(\mathcal O)))
      }.
    \end{aligned}
  \end{equation}
\end{lemma}
\begin{proof}
  We have
  \begin{align*} 
    & \sum_{j=0}^{J-1} \Big[
      \int_{t_j}^{t_{j+1}} (\mathcal S_2f + f)(t) \, \mathrm{d}t, \,
      (\mathcal S_0g)_j \delta W_j
    \Big] \\
    \leqslant{} &
    \sum_{j=0}^{J-1} \Nm{
      \int_{t_j}^{t_{j+1}} (\mathcal S_2 f + f)(t) \, \mathrm{d}t
    }_{L^2(\Omega;L^2(\mathcal O))} \nm{
      (\mathcal S_0g)_j \delta W_j
    }_{L^2(\Omega;L^2(\mathcal O))} \\
    \leqslant{} &
    \tau^{1/2} \sum_{j=0}^{J-1} \nm{\mathcal S_2f + f}_{
      L^2(\Omega;L^2(t_j,t_{j+1};L^2(\mathcal O)))
    } \nm{(\mathcal S_0g)_j \delta W_j}_{
      L^2(\Omega;L^2(\mathcal O))
    } \\
    ={} &
    \tau \sum_{j=0}^{J-1} \nm{\mathcal S_2f + f}_{
      L^2(\Omega;L^2(t_j,t_{j+1};L^2(\mathcal O)))
    } \nm{(\mathcal S_0g)_j}_{
      L^2(\Omega;L^2(\mathcal O))
    } \\
    \leqslant{} &
    \tau^{1/2} \nm{\mathcal S_2f + f}_{
      L^2(\Omega;L^2(0,T;L^2(\mathcal O)))
    } \nm{\mathcal S_0g}_{
      L^2(\Omega;L^2(0,T;L^2(\mathcal O)))
    },
  \end{align*}
  so that \cref{eq:ljm} follows from \cref{lem:S0-stab,lem:S2-stab}.
\end{proof}

\begin{lemma}
  \label{lem:y-wtY}
  Let $ u $ be the solution of problem \cref{eq:model} with
  \[
    y_d \in L_\mathbb F^2(\Omega;L^2(0,T;L^2(\mathcal O))).
  \]
  Then
  \begin{equation} 
    \label{eq:S0-wtU}
    \begin{aligned}
      & \nm{S_0u-\mathcal S_0\widetilde U}_{
        L^2(\Omega;L^2(0,T;L^2(\mathcal O)))
      } \\
      \lesssim{} &
      \big( (1+\nu^{-1}) \tau^{1/2} + h^2 \big) \big(
        \nm{y_d}_{L^2(\Omega;L^2(0,T;L^2(\mathcal O)))} +
        \nm{u}_{L^2(\Omega;L^2(0,T;L^2(\mathcal O)))}
      \big),
    \end{aligned}
  \end{equation}
  where $ \widetilde U \in U_\text{ad}^{h,\tau} $ is defined by
  \begin{equation} 
    \label{eq:wtU}
    \widetilde U_j := \mathbb E_{t_j} \Big(
      \frac1\tau \int_{t_j}^{t_{j+1}} u(t) \, \mathrm{d}t
    \Big), \quad 0 \leqslant j < J.
  \end{equation}
\end{lemma}
\begin{proof}
  Since
  \begin{align*} 
    & \nm{S_0u - \mathcal S_0\widetilde U}_{
      L^2(\Omega;L^2(0,T;L^2(\mathcal O)))
    } \\
    ={} &
    \nm{S_0(u-\widetilde U) + (S_0 - \mathcal S_0)\widetilde U}_{
      L^2(\Omega;L^2(0,T;L^2(\mathcal O)))
    } \\
    \leqslant{} &
    \nm{S_0(u-\widetilde U)}_{
      L^2(\Omega;L^2(0,T;L^2(\mathcal O)))
    } +
    \nm{(S_0 - \mathcal S_0)\widetilde U}_{
      L^2(\Omega;L^2(0,T;L^2(\mathcal O)))
    } \\
    \lesssim{} &
    \nm{u-\widetilde U}_{L^2(\Omega;L^2(0,T;L^2(\mathcal O)))} +
    \nm{(S_0-\mathcal S_0)\widetilde U}_{
      L^2(\Omega;L^2(0,T;L^2(\mathcal O)))
    } \quad\text{(by \cref{eq:S0-regu})} \\
    \lesssim{} &
    \nm{u-\widetilde U}_{L^2(\Omega;L^2(0,T;L^2(\mathcal O)))} +
    (\tau^{1/2} + h^2)\nm{\widetilde U}_{
      L^2(\Omega;L^2(0,T;L^2(\mathcal O)))
    } \quad\text{(by \cref{eq:S0-conv})} \\
    \lesssim{} &
    \nm{u-\widetilde U}_{L^2(\Omega;L^2(0,T;L^2(\mathcal O)))} +
    (\tau^{1/2} + h^2)\nm{u}_{
      L^2(\Omega;L^2(0,T;L^2(\mathcal O)))
    } \quad\text{(by \cref{eq:wtU})},
  \end{align*}
  it suffices to prove
  \begin{equation}
    \label{eq:u-wtu}
    \begin{aligned}
      & \nm{u-\widetilde U}_{L^2(\Omega;L^2(0,T;L^2(\mathcal O)))} \\
      \lesssim{} &
      \tau^{1/2} \nu^{-1} \big(
        \nm{u}_{L^2(\Omega;L^2(0,T;L^2(\mathcal O)))} +
        \nm{y_d}_{L^2(\Omega;L^2(0,T;L^2(\mathcal O)))}
      \big).
    \end{aligned}
  \end{equation}
  To this end, we proceed as follows. By \cref{thm:optim} we have
  \[ 
    u = P_{U_\text{ad}} (-\nu^{-1} p),
  \]
  where $ P_{U_\text{ad}} $ is the $ L^2(\mathcal O) $-orthogonal projection
  onto $ U_\text{ad} $ and $ p := S_1(S_0u - y_d) $. By the evident inequality
  (cf.~\cite[Lemma 1.10]{Hinze2009})
  \[
    \nm{P_{U_\text{ad}} v - P_{U_\text{ad}}w}_{L^2(\mathcal O)}
    \leqslant \nm{v-w}_{L^2(\mathcal O)} \quad
    \forall v, w \in L^2(\mathcal O),
  \]
  we then obtain
  \begin{align*} 
    & \sum_{j=0}^{J-1} \nm{ u-u(t_j) }_{
      L^2(\Omega;L^2(t_j,t_{j+1};L^2(\mathcal O)))
    }^2 \\
    \leqslant{} &
    \nu^{-2} \sum_{j=0}^{J-1} \nm{p-p(t_j)}_{
      L^2(\Omega;L^2(t_j,t_{j+1};L^2(\mathcal O)))
    }^2 \\
    \lesssim{} &
    \tau \nu^{-2} \nm{S_0u-y_d}_{
      L^2(\Omega;L^2(0,T;L^2(\mathcal O)))
    }^2 \quad\text{(by \cref{eq:p-pj})} \\
    \lesssim{} &
    \tau \nu^{-2} \big(
      \nm{y_d}_{L^2(\Omega;L^2(0,T;L^2(\mathcal O)))}^2 +
      \nm{u}_{L^2(\Omega;L^2(0,T;L^2(\mathcal O)))}^2
    \big) \quad\text{(by \cref{eq:S0-regu}).}
  \end{align*}
  Hence, \cref{eq:u-wtu} follows from the inequality
  \begin{equation}
    \label{eq:kdkkdk}
    \sum_{j=0}^{J-1} \nm{u-\widetilde U}_{
      L^2(\Omega;L^2(t_j,t_{j+1};L^2(\mathcal O)))
    }^2 \leqslant \sum_{j=0}^{J-1}
    \nm{u-u(t_j)}_{
      L^2(\Omega;L^2(t_j,t_{j+1};L^2(\mathcal O)))
    }^2.
  \end{equation}
  The above inequality is a direct consequence of \cref{eq:wtU}. Indeed, by
  \cref{eq:wtU} we have
  \[
    \int_{t_j}^{t_{j+1}} [u-\widetilde U, \widetilde U] \, \mathrm{d}t =
    \int_{t_j}^{t_{j+1}} [u-\widetilde U, u(t_j)] \, \mathrm{d}t = 0,
    \quad 0 \leqslant j < J,
  \]
  and so
  \begin{small}
  \begin{align*}
    & \nm{u-\widetilde U}_{L^2(\Omega;L^2(t_j,t_{j+1};L^2(\mathcal O)))}^2
    = \int_{t_j}^{t_{j+1}} [u-\widetilde U, u - \widetilde U] \, \mathrm{d}t
    =
    \int_{t_j}^{t_{j+1}} [u-\widetilde U, u - u(t_j)] \, \mathrm{d}t \\
    \leqslant{} &
    \nm{u-\widetilde U}_{L^2(\Omega;L^2(t_j,t_{j+1};L^2(\mathcal O)))}
    \nm{u-u(t_j)}_{
      L^2(\Omega;L^2(t_j,t_{j+1};L^2(\mathcal O)))
    }, \quad 0 \leqslant j < J,
  \end{align*}
  \end{small}
  It follows that
  \[
    \nm{u-\widetilde U}_{L^2(\Omega;L^2(t_j,t_{j+1};L^2(\mathcal O)))}
    \leqslant \nm{u-u(t_j)}_{L^2(\Omega;L^2(t_j,t_{j+1};L^2(\mathcal O)))},
    \, 0 \leqslant j < J,
  \]
  which implies \cref{eq:kdkkdk}. This completes the proof.
\end{proof}

Then, we are in a position to conclude the proof of \cref{thm:conv} as follows.

\medskip\noindent{\bf Proof of \cref{thm:conv}.} For convenience, in this proof
  we will use the following convention: $ C $ denotes a positive constant
  depending only on $ \nu $, $ u_* $, $ u^* $, $ y_d $, $ T $, $ \mathcal O $,
  and the regularity parameters of $ \mathcal K_h $, and its value may differ at
  each occurrence. Let $ \widetilde U $ be defined by \cref{eq:wtU}. By
  definition we have the following two evident equalities:
  \begin{equation}
    \label{eq:250}
    \int_0^T [U, \widetilde U - U] \, \mathrm{d}t =
    \int_0^T [U, u-U] \, \mathrm{d}t
  \end{equation}
  and
  \begin{equation}
    \label{eq:249}
    \begin{aligned}
      & \nm{u-U}_{L^2(\Omega;L^2(0,T;L^2(\mathcal O)))}^2 \\
      ={} &
      \nm{u-\widetilde U}_{L^2(\Omega;L^2(0,T;L^2(\mathcal O)))}^2 +
      \nm{U-\widetilde U}_{L^2(\Omega;L^2(0,T;L^2(\mathcal O)))}^2.
    \end{aligned}
  \end{equation}
  The rest of the proof is divided into the following three steps.

  {\it Step 1.} Let us prove
  \begin{small}
  \begin{equation}
    \label{eq:0}
    \frac12 \nm{S_0u-\mathcal S_0U}_{L^2(\Omega;L^2(0,T;L^2(\mathcal O)))}^2 +
    \nu \nm{u-U}_{L^2(\Omega;L^2(0,T;L^2(\mathcal O)))}^2 \leqslant
    \mathbb I_1 + \mathbb I_2 + \mathbb I_3,
  \end{equation}
  \end{small}
  where
  \begin{small}
  \begin{align}
    \mathbb I_1 &:=
    \frac12\nm{S_0u-\mathcal S_0 \widetilde U}_{
      L^2(\Omega;L^2(0,T;L^2(\mathcal O)))
    }^2, \label{eq:I1} \\
    \mathbb I_2 &=
    \sum_{j=0}^{J-1} \Big[
      \int_{t_j}^{t_{j+1}}
      \big( (\mathcal S_2 + I)(y_d - S_0u) \big)(t) \, \mathrm{d}t, \,
      \big( \mathcal S_0(\widetilde U - U) \big)_j \delta W_j
    \Big], \label{eq:I2} \\
    \mathbb I_3 &:=
    \int_0^T \big[ S_1(S_0u-y_d), U\!-\!u \big] \, \mathrm{d}t -
    \sum_{j=0}^{J-1} \int_{t_j}^{t_{j+1}}
    \big[
      \big( \mathcal S_1(S_0u-y_d) \big)_{j+1}, U-\widetilde U
    \big] \, \mathrm{d}t.
    \label{eq:I3}
  \end{align}
  \end{small}
  Since $ U $ is the solution of problem \cref{eq:discretization},
  by \cref{lem:S0f-g} we deduce that
  \begin{equation} 
    \label{eq:discr-optim}
    \sum_{j=0}^{J-1}
    \int_{t_j}^{t_{j+1}} \big[
      \big(
        \mathcal S_1(\mathcal S_0U-y_d)
      \big)_{j+1}, \, \widetilde U-U
    \big] \, \mathrm{d}t +
    \nu \int_0^T [U, \widetilde U - U] \, \mathrm{d}t +
    \mathcal A_0 \geqslant 0,
  \end{equation}
  where
  \begin{equation}
    \label{eq:calA0}
    \mathcal A_0 := -\sum_{j=0}^{J-1} \Big[
      \int_{t_j}^{t_{j+1}} \big(
        (\mathcal S_2+I)(\mathcal S_0U-y_d)
      \big)(t) \, \mathrm{d}t, \,
      (\mathcal S_0(\widetilde U-U))_j \delta W_j
    \Big].
  \end{equation}
  By \cref{thm:optim} we get
  \begin{equation} 
    \label{eq:P-wtP0}
    \int_0^T [S_1(S_0u-y_d) + \nu u, \, U-u] \, \mathrm{d}t \geqslant 0.
  \end{equation}
  Combining \cref{eq:discr-optim,eq:P-wtP0} yields
  \begin{small}
  \begin{align*} 
    & \nu \int_0^T [u,u-U] - [U, \widetilde U-U] \, \mathrm{d}t  \\
    \leqslant{} &
    \int_0^T [S_1(S_0u-y_d), U-u] \, \mathrm{d}t +
    \sum_{j=0}^{J-1} \int_{t_j}^{t_{j+1}} \big[
      (\mathcal S_1(\mathcal S_0U-y_d))_{j+1}, \, \widetilde U - U
    \big] \, \mathrm{d}t + \mathcal A_0 \\
    \leqslant{} &
    \mathbb I_3 + \sum_{j=0}^{J-1} \int_{t_j}^{t_{j+1}}
    \big[
      (\mathcal S_1(S_0u-\mathcal S_0U))_{j+1}, U - \widetilde U
    \big] + \mathcal A_0 \quad\text{(by \cref{eq:I3}),}
  \end{align*}
  \end{small}
  so that by \cref{eq:250} we obtain
  \begin{small}
  \begin{equation}
    \label{eq:sj}
    \nu \nm{u-U}_{L^2(\Omega;L^2(0,T;L^2(\mathcal O)))}^2
    \leqslant
    \mathbb I_3 + \sum_{j=0}^{J-1} \int_{t_j}^{t_{j+1}}
    \big[
      (\mathcal S_1(S_0u-\mathcal S_0U))_{j+1}, U - \widetilde U
    \big] + \mathcal A_0.
  \end{equation}
  \end{small}
  We also have
  \begin{align*} 
    & \sum_{j=0}^{J-1} \int_{t_j}^{t_{j+1}} \big[
      (\mathcal S_1(S_0u-\mathcal S_0U))_{j+1},
      \, U - \widetilde U
    \big] \, \mathrm{d}t \\
    ={} &
    \int_0^T [S_0u - \mathcal S_0U,
    \, \mathcal S_0 (U-\widetilde U)] \, \mathrm{d}t +
    \mathcal A_1 \quad\text{(by \cref{eq:S1f-g})} \\
    ={} &
    -\nm{S_0u-\mathcal S_0U}_{L^2(\Omega;L^2(0,T;L^2(\mathcal O)))}^2 +
    \int_0^T [S_0u-\mathcal S_0U, S_0u-\mathcal S_0\widetilde U] \, \mathrm{d}t +
    \mathcal A_1 \\
    \leqslant{} &
    -\frac12 \nm{S_0u-\mathcal S_0U}_{L^2(\Omega;L^2(0,T;L^2(\mathcal O)))}^2 +
    \frac12 \nm{S_0u-\mathcal S_0\widetilde U}_{
      L^2(\Omega;L^2(0,T;L^2(\mathcal O)))
    }^2 + \mathcal A_1 \\
    ={} &
    -\frac12 \nm{S_0u-\mathcal S_0U}_{L^2(\Omega;L^2(0,T;L^2(\mathcal O)))}^2 +
    \mathbb I_1 + \mathcal A_1 \quad\text{(by \cref{eq:I1}),}
  \end{align*}
  where
  \begin{equation}
    \label{eq:calA1}
    \mathcal A_1 := \sum_{j=0}^{J-1} \Big[
      \int_{t_j}^{t_{j+1}} \big(
        (\mathcal S_2 + I)(S_0u-\mathcal S_0U)
      \big)(t) \, \mathrm{d}t, \,
      ( \mathcal S_0(U-\widetilde U) )_j\delta W_j
    \Big].
  \end{equation}
  Finally, inserting the above estimate into \cref{eq:sj} yields
  \begin{align*}
    & \nu \nm{u-U}_{L^2(\Omega;L^2(0,T;L^2(\mathcal O)))}^2 \\
    \leqslant{} &
    \mathbb I_3 - \frac12 \nm{S_0u-\mathcal S_0U}_{
      L^2(\Omega;L^2(0,T;L^2(\mathcal O)))
    }^2 + \mathbb I_1 +  \mathcal A_1 + \mathcal A_0 \\
    ={} &
    \mathbb I_3 - \frac12 \nm{S_0u-\mathcal S_0U}_{
      L^2(\Omega;L^2(0,T;L^2(\mathcal O)))
    }^2 + \mathbb I_1 +  \mathbb I_2
    \quad\text{(by \cref{eq:I2,eq:calA0,eq:calA1}),}
  \end{align*}
  which indicates the desired inequality \cref{eq:0}.

  {\it Step 2.} Let us estimate $ \mathbb I_1 $, $ \mathbb I_2 $ and $ \mathbb
  I_3 $. For $ \mathbb I_1 $, by \cref{lem:y-wtY} and the fact $ u \in
  U_\text{ad} $ we have
  \begin{equation}
    \label{eq:22}
    \mathbb I_1 \leqslant C (\tau^{1/2} + h^2)^2.
  \end{equation}
  For $ \mathbb I_2 $, from \cref{lem:ljm,eq:249} we deduce
  \begin{align*} 
    \mathbb I_2 & \leqslant C \tau^{1/2}
    \nm{y_d - S_0u}_{L^2(\Omega;L^2(0,T;L^2(\mathcal O)))}
    \nm{\widetilde U - U}_{
      L^2(\Omega;L^2(0,T;L^2(\mathcal O)))
    } \\
    & \leqslant C \tau^{1/2}
    \nm{y_d - S_0u}_{L^2(\Omega;L^2(0,T;L^2(\mathcal O)))}
    \nm{u-U}_{
      L^2(\Omega;L^2(0,T;L^2(\mathcal O)))
    }.
  \end{align*}
  Since \cref{eq:S0-regu} and the fact $ u \in U_\text{ad} $ imply
  \begin{equation}
    \label{eq:3000}
    \nm{y_d - S_0u}_{L^2(\Omega;L^2(0,T;L^2(\mathcal O)))}
    \leqslant C,
  \end{equation}
  it follows that
  \begin{equation}
    \label{eq:1}
    \mathbb I_2 \leqslant C \tau^{1/2} \nm{u-U}_{
      L^2(\Omega;L^2(0,T;L^2(\mathcal O)))
    }.
  \end{equation}
  Now let us estimate $ \mathbb I_3 $. By \cref{eq:S1-conv-new,eq:249,eq:3000}
  we get
  \begin{align*}
    & \sum_{j=0}^{J-1} \int_{t_j}^{t_{j+1}} \big[
      S_1(S_0u-y_d) - (\mathcal S_1(S_0u-y_d))_{j+1},
      \, U-\widetilde U
    \big] \, \mathrm{d}t \\
    \leqslant{} &
    C(\tau^{1/2} + h^2) \nm{u-U}_{
      L^2(\Omega;L^2(0,T;L^2(\mathcal O)))
    }
  \end{align*}
  and by \cref{eq:p-pj,eq:249,eq:3000} we get
  \begin{align*} 
    & \sum_{j=0}^{J-1} \int_{t_j}^{t_{j+1}}
    \Big[
      \big(S_1(S_0u-y_d)\big)(t) -
      \big(S_1(S_0u-y_d)\big)(t_j),
      \, \widetilde U - u
    \Big] \, \mathrm{d}t \\
    \leqslant{} &
    C \tau^{1/2} \nm{u - U}_{
      L^2(\Omega;L^2(0,T;L^2(\mathcal O)))
    }.
  \end{align*}
  Consequently,
  \begin{small}
  \begin{align} 
    \mathbb I_3 &=
    \sum_{j=0}^{J-1} \int_{t_j}^{t_{j+1}}
    \big[
      S_1(S_0u-y_d) - (\mathcal S_1(S_0u-y_d))_{j+1}, \,
      U - \widetilde U
    \big] \, \mathrm{d}t \notag \\
    & \qquad\qquad{} +
    \int_0^T [S_1(S_0u-y_d), \widetilde U - u] \, \mathrm{d}t
    \quad\text{(by \cref{eq:I3})} \notag \\
    &=
    \sum_{j=0}^{J-1} \int_{t_j}^{t_{j+1}}
    \big[
      S_1(S_0u-y_d) - (\mathcal S_1(S_0u-y_d))_{j+1}, \,
      U - \widetilde U
    \big] \, \mathrm{d}t \notag \\
    & \qquad\qquad{} +
    \sum_{j=0}^{J-1} \int_{t_j}^{t_{j+1}} [
      S_1(S_0u-y_d) - (S_1(S_0u-y_d))(t_j) , \widetilde U - u
    ] \, \mathrm{d}t \quad\text{(by \cref{eq:wtU})} \notag \\
    & \leqslant C (\tau^{1/2} + h^2)
    \nm{u-U}_{
      L^2(\Omega;L^2(0,T;L^2(\mathcal O)))
    }. \label{eq:3}
  \end{align}
  \end{small}

  {\it Step 3.} Combining \cref{eq:0,eq:22,eq:1,eq:3} yields
  \begin{align*} 
    & \frac12 \nm{S_0u-\mathcal S_0U}_{
      L^2(\Omega;L^2(0,T;L^2(\mathcal O)))
    }^2 + \nu \nm{u-U}_{
      L^2(\Omega;L^2(0,T;L^2(\mathcal O)))
    }^2 \\
    \leqslant{} & C (\tau^{1/2} + h^2) \big(
      \tau^{1/2}+h^2 +
      \nm{u-U}_{L^2(\Omega;L^2(0,T;L^2(\mathcal O)))}
    \big).
  \end{align*}
  Applying the Young's inequality with $ \epsilon $ to the above inequality then
  gives
  \begin{align*}
    & \nm{S_0u-\mathcal S_0U}_{
      L^2(\Omega;L^2(0,T;L^2(\mathcal O)))
    }^2 + \nm{u-U}_{
      L^2(\Omega;L^2(0,T;L^2(\mathcal O)))
    }^2 \\
    \leqslant{} & C(\tau^{1/2} + h^2)^2,
  \end{align*}
  which implies the desired inequality \cref{eq:conv}. This completes the proof.
\hfill\ensuremath{\blacksquare}

\begin{remark}
  We would like to give a brief proof of \cref{eq:discr-optim} as follows. A
  straightforward computation gives
  \begin{align*}
    & \lim_{s \to 0+} \frac{
      J_{h,\tau}(U + s(\widetilde U-U)) - J_{h,\tau}(U)
    } s \\
    ={} &
    \int_0^T [\mathcal S_0 U - y_d, \, \mathcal S_0(\widetilde U-U)] \, \mathrm{d}t +
    \nu \int_0^T [U, \widetilde U-U] \, \mathrm{d}t \\
    ={} &
    \sum_{j=0}^{J-1} \int_{t_j}^{t_{j+1}}
    [\widetilde U-U, (\mathcal S_1(\mathcal S_0U-y_d))_{j+1}] \, \mathrm{d}t \\
    & \quad {} - \sum_{j=0}^{J-1}
    \big[
      (\mathcal S_0(\widetilde U-U))_j \delta W_j,
      \int_{t_j}^{t_{j+1}} (\mathcal S_2 + I)(\mathcal S_0U - y_d) \, \mathrm{d}t
    \big] \\
    & \quad{} + \nu \int_0^T [U, \widetilde U - U] \, \mathrm{d}t
    \quad\text{\text{(by \cref{lem:S0f-g})}} \\
    ={} &
    \sum_{j=0}^{J-1} \int_{t_j}^{t_{j+1}}
    [(\mathcal S_1(\mathcal S_0U - y_d))_{j+1}, \widetilde U - U] \, \mathrm{d}t
    + \nu \int_0^T [U, \widetilde U - U] \, \mathrm{d}t +
    \mathcal A_0,
  \end{align*}
  by the definition of $ \mathcal A_0 $. Since $ U $ is the unique solution of
  problem \cref{eq:discretization}, we then obtain
  \[
    \sum_{j=0}^{J-1} \int_{t_j}^{t_{j+1}}
    [(\mathcal S_1(\mathcal S_0U - y_d))_{j+1}, \widetilde U - U] \, \mathrm{d}t
    + \nu \int_0^T [U, \widetilde U - U] \, \mathrm{d}t +
    \mathcal A_0 \geqslant 0,
  \]
  namely equation \cref{eq:discr-optim}.
\end{remark}






\appendix
\section{Proof of \texorpdfstring{\cref{lem:lbj}}{}}
\label{sec:lbj}
Since \cref{eq:xx-2} can be proved by the same argument as that used in the
proof of \cite[Theorem 5.1]{Vexler2008I}, we only prove \cref{eq:xx-1}, using
the techniques in the proof of \cite[Theorem 12.1]{Thomee2006}. Let
\[ 
  \theta_j := P_j - \eta(t_j), \quad 0 \leqslant j \leqslant J.
\]
By \cref{eq:2074} we have
\begin{equation}
  \label{eq:2075}
  \begin{aligned}
    & \sum_{j=0}^{J-1} [P_j - P_{j+1}, (-\Delta_h)^{-1} \theta_j]
    + \sum_{j=0}^{J-1} \tau [P_j, \theta_j] \\
    ={} &
    \sum_{j=0}^{J-1} \int_{t_j}^{t_{j+1}} [Q_hg(t), (-\Delta_h)^{-1} \theta_j]
    \, \mathrm{d}t.
  \end{aligned}
\end{equation}
Since \cref{eq:2073} implies
\begin{equation}
  \begin{cases}
    -\eta'(t) - \Delta_h \eta(t) = Q_h g(t),
    \quad 0 \leqslant t \leqslant T, \\
    \eta(T) = 0,
  \end{cases}
\end{equation}
we obtain
\begin{align*}
  & \sum_{j=0}^{J-1} \int_{t_j}^{t_{j+1}}
  [-\eta'(t), (-\Delta_h)^{-1}\theta_j] \, \mathrm{d}t +
  \sum_{j=0}^{J-1} \int_{t_j}^{t_{j+1}} [\eta(t), \theta_j] \, \mathrm{d}t \\
  ={} &
  \sum_{j=0}^{J-1} \int_{t_j}^{t_{j+1}} [Q_hg(t), (-\Delta_h)^{-1} \theta_j]
  \, \mathrm{d}t.
\end{align*}
By integration by parts, we then obtain
\begin{equation}
  \label{eq:2076}
  \begin{aligned}
    & \sum_{j=0}^{J-1} [\eta(t_j) - \eta(t_{j+1}), (-\Delta_h)^{-1}\theta_j] +
    \sum_{j=0}^{J-1} \int_{t_j}^{t_{j+1}} [\eta(t), \theta_j] \, \mathrm{d}t \\
    ={} &
    \sum_{j=0}^{J-1} \int_{t_j}^{t_{j+1}} [Q_hg(t), (-\Delta_h)^{-1} \theta_j]
    \, \mathrm{d}t.
  \end{aligned}
\end{equation}
Combining \cref{eq:2075,eq:2076} yields
\begin{align*}
  \sum_{j=0}^{J-1}
  [\theta_j - \theta_{j+1}, (-\Delta_h)^{-1} \theta_j] +
  \sum_{j=0}^{J-1} \int_{t_j}^{t_{j+1}}
  [P_j - \eta(t), \theta_j] \, \mathrm{d}t = 0.
\end{align*}
Since
\begin{align*}
  \sum_{j=0}^{J-1} [\theta_j - \theta_{j+1}, (-\Delta_h)^{-1}\theta_j]
  & = \sum_{j=0}^{J-1} \nm{\theta_j}_{\dot H_h^{-1}(\mathcal O)}^2 -
  [(-\Delta_h)^{-1/2}\theta_j, (-\Delta_h)^{-1/2} \theta_{j+1}] \\
  & \geqslant
  \sum_{j=0}^{J-1} \big(
    \nm{\theta_j}_{\dot H_h^{-1}(\mathcal O)}^2 -
    \nm{\theta_j}_{\dot H_h^{-1}(\mathcal O)}
    \nm{\theta_{j+1}}_{\dot H_h^{-1}(\mathcal O)}
  \big) \\
  & \geqslant
  \frac12 \nm{\theta_0}_{\dot H_h^{-1}(\mathcal O)}^2 -
  \frac12 \nm{\theta_J}_{\dot H_h^{-1}(\mathcal O)}^2 \\
  & =
  \frac12 \nm{\theta_0}_{\dot H_h^{-1}(\mathcal O)}^2
  \quad\text{(by the fact $ P_J = \eta(t_J) = 0 $),}
\end{align*}
it follows that
\[
  \frac12 \nm{\theta_0}_{\dot H_h^{-1}(\mathcal O)}^2 +
  \sum_{j=0}^{J-1} \int_{t_j}^{t_{j+1}}
  [P_j - \eta(t), \theta_j] \, \mathrm{d}t \leqslant 0.
\]
Hence,
\begin{align*}
  & \frac12 \nm{\theta_0}_{\dot H_h^{-1}(\mathcal O)}^2 +
  \sum_{j=0}^{J-1} \tau \nm{\theta_j}_{L^2(\mathcal O)}^2 \\
  \leqslant{} &
  \sum_{j=0}^{J-1} \int_{t_j}^{t_{j+1}}
  [\eta(t) - \eta(t_j), \theta_j] \, \mathrm{d}t \\
  \leqslant{} &
  \Big(
    \sum_{j=0}^{J-1} \nm{\eta - \eta(t_j)}_{
      L^2(t_j,t_{j+1};L^2(\mathcal O))
    }^2
  \Big)^{1/2} \Big(
    \sum_{j=0}^{J-1} \tau \nm{\theta_j}_{L^2(\mathcal O)}^2
  \Big)^{1/2} \\
  \leqslant{} &
  \frac12\sum_{j=0}^{J-1} \nm{\eta - \eta(t_j)}_{
    L^2(t_j,t_{j+1};L^2(\mathcal O))
  }^2 + \frac12 \sum_{j=0}^{J-1} \tau \nm{\theta_j}_{L^2(\mathcal O)}^2,
\end{align*}
which implies
\begin{align*}
  \nm{\theta_0}_{\dot H_h^{-1}(\mathcal O)}
  & \leqslant \Big(
    \sum_{j=0}^{J-1} \nm{\eta - \eta(t_j)}_{L^2(t_j,t_{j+1};L^2(\mathcal O))}^2
  \Big)^{1/2}.
\end{align*}
Therefore, by the standard estimate
\begin{align*}
  \sum_{j=0}^{J-1} \nm{\eta - \eta(t_j)}_{L^2(\mathcal O)}^2
  \lesssim \tau^2 \nm{\eta'}_{L^2(t_j,t_{j+1};L^2(\mathcal O))}^2
  \lesssim \tau^2 \nm{g}_{L^2(0,T;L^2(\mathcal O))}^2,
\end{align*}
we obtain
\[
  \nm{\theta_0}_{\dot H_h^{-1}(\mathcal O)}
  \lesssim \tau \nm{g}_{L^2(0,T;L^2(\mathcal O))},
\]
namely,
\[
  \nm{\eta(0) - P_0}_{\dot H_h^{-1}(\mathcal O)}
  \lesssim \tau \nm{g}_{L^2(0,T;L^2(\mathcal O))}.
\]
Similarly, we can prove
\[
  \nm{\eta(t_j) - P_j}_{\dot H_h^{-1}(\mathcal O)}
  \lesssim \tau \nm{g}_{L^2(0,T;L^2(\mathcal O))},
  \quad 1 \leqslant j < J.
\]
Since $ P_J = \eta(t_J) = 0 $, this proves \cref{eq:xx-1} and thus completes
the proof.

\bibliographystyle{plain}

\end{document}